\documentclass[11pt]{article}
\usepackage{graphicx,graphics}

\usepackage{latexsym}
\usepackage{caption}
\usepackage{amsfonts}
\usepackage{amsmath,amsthm,enumerate, mathrsfs, amssymb, indentfirst,ulem}
\usepackage{enumerate}
\usepackage{enumitem}
\usepackage{authblk}
\usepackage{pict2e} 

\usepackage{subfig}
\usepackage{pgfplots}
\usepgfplotslibrary{colormaps, external}
\usepackage{float}
\usepackage{cite}
\usepackage{tikz,tkz-graph}
\usetikzlibrary{calc}
\usetikzlibrary{backgrounds, shapes.geometric}
\usepackage{rotating}
\usetikzlibrary{positioning}
\usepackage{pgffor}

\usepackage{pgfplots}
\pgfplotsset{compat=1.15}

\newtheorem{theorem}{Theorem}

\newtheorem{lemma}[theorem]{Lemma}

\newtheorem{observation}[theorem]{Observation}
\newtheorem{proposition}[theorem]{Proposition}

\usepackage{todonotes}

\definecolor{lightblue}{rgb}{0.68,0.85,0.9}

\author[1]{Anna Gujgiczer}
\author[2]{Reza Naserasr} 
\author[3]{Rohini S}
\author[4]{S Taruni}

\affil[1]{Department of Computer Science and Information Theory, Faculty of Electrical Engineering and Informatics,
	Budapest University of Technology and Economics and
	MTA-BME Lend\"ulet Arithmetic Combinatorics Research Group
	Email address: gujgicza@cs.bme.hu}
\affil[2]{Universit\'{e} de Paris, IRIF, CNRS, F-75006, Paris, France. 	
	Email address: reza@irif.fr.}	
\affil[3]{Department of Mathematics
	Indian Institute of Technology Madras
	Chennai 600 036, India
	Email: s.rohini@smail.iitm.ac.in}
\affil[4]{Department of Mathematics, Indian Institute of Technology Dharwad, Dharwad 580011, India Email address : taruni.sridhar@gmail.com}

\begin{document}
	\setlength{\parindent}{0pt}	
	
	\title{Winding number and circular 4-coloring of signed graphs}	
	\date{\today} 
	
	\maketitle 
	\begin{abstract}
		Concerning the recent notion of circular chromatic number of signed graphs, for each given integer $k$ we introduce two signed bipartite graphs, each on $2k^2-k+1$ vertices, having shortest negative cycle of length $2k$, and the circular chromatic number 4. 
		
		Each of the construction can be viewed as a bipartite analogue of the generalized Mycielski graphs on odd cycles, $M_{\ell}(C_{2k+1})$. In the course of proving our result, we also obtain a simple proof of the fact that $M_{\ell}(C_{2k+1})$ and some similar quadrangulations of the projective plane have circular chromatic number 4. These proofs have the advantage that they illuminate, in an elementary manner, the strong relation between algebraic topology and graph coloring problems.
		
	\end{abstract}

	\section{Introduction}

	The question of the chromatic number of graphs embedded on a given surface is one of the fundamental questions in graph theory which has led to the development of this subject on various directions such as minor theory, coloring and homomorphism in general, with the relation between the two notions of minor and coloring being one of the most central part of graph theory.
	
	To better understand this relation the notion of balanced-chromatic number of signed graphs has been recently introduced in \cite{JMNNQ23} using which a signed version of the Hadwiger conjecture is presented. The circular chromatic number of signed graphs, introduced earlier in \cite{NWZ21}, is observed to be a refinement of the balanced-chromatic number of signed graphs. Motivated by these developments and in connection to several other work, in this paper we study the question of circular chromatic number of signed graphs embedded on the projective plane such that every face is a positive 4-cycle. Notation and terminology are given in Section~\ref{sec:Notation}.
	  
	\subsection{Further motivation}
	The problem of building graphs of high girth and high chromatic number is one of the basic questions of graph coloring and it has been studied extensively in the literature. In particular, the original proof of Erd\H{o}s for the existence of such graphs has led to the development of probabilistic methods in graph theory. Since then several constructive methods were presented, but none are easy to grasp. With a weaker condition of high odd girth instead of high girth, there are several natural classes of graphs. In particular, in the family of the Kneser graphs one can find examples of high odd girth and high chromatic number. The proof of the lower bound for the chromatic number of the Kneser graphs, by L. Lov\'asz \cite{LLKn}, was the birthplace of the connection between algebraic topology and graph coloring. Further developing this method, Stiebitz introduced a generalization of the Mycielski construction in \cite{S85} to build small graphs of high odd girth and high chromatic number. 
	Generalized Mycielski on odd cycles have been studied independently by many authors and several results on their chromatic number \cite{P92, VT95, Y96}, circular chromatic number \cite{CHZ99, DGMVZ05, HZ03, LLGS03, ST06} and on various other related parameters \cite{lin2006several, T01} are proved. These graphs hold the best known upper bound on the order of graphs of given odd-girth and chromatic number 4 and their orders are not far from the best known lower bounds.
	
	In this work, building on the ideas from several works in the literature, we first present a relatively short proof that the generalized Mycielski graphs on odd cycles have circular chromatic number 4. The proof has the advantage of formalizing the use of winding number in graph coloring (see Section~\ref{sec:Conclusions}). We then present three similar classes of signed graphs of high negative girth and circular chromatic number 4. The graphs are built similarly to the generalized Mycielski on odd cycles when viewed as a quadrangulation of the projective plane, the main difference being that the subgraph induced by the outer layer induces a M\"obius ladder. In particular, for each positive integer $k$, we have two signed bipartite graphs, each on $2k^2-k+1$ vertices, where the shortest negative cycles is of length $2k$ and that their circular chromatic number is $4$. This is conjectured to be the best possible, see Section~\ref{sec:Construction} for more.

	In Section~\ref{sec:Notation}, we give the necessary notation and the terminology. In Section~\ref{sec:history}, we provide a historical account of what is known. In Section~\ref{sec:Construction}, we discuss three families of signed graphs and in the Section~\ref{sec:Winding-Coloring}, we prove that their circular chromatic number is $4$. Concluding remarks are given in Section~\ref{sec:Conclusions}. 
	
	\section{Notation}\label{sec:Notation}

	We consider simple graphs unless clearly stated otherwise. A \textit{signed (simple) graph} $(G,\sigma)$ is a graph $G$ together with the assignment $\sigma$ of signs to the edges. We denote by $(G, -)$ the signed graph $G$ with all edges negative. If $G$ is bipartite, then  $(G, \sigma)$ is called a \textit{signed bipartite graph} (in some literature, this term is used to refer to a balanced signed graph, that is a signed graph with no negative cycle). The sign of a structure in $(G, \sigma)$ (such as a cycle, a closed walk, a path) is the product of the signs of edges in the said structure counting multiplicity. 
	
	Given an integer $n$, $n\geq 3$, we denote by $C_{n}$ the cycle (graph) on $n$ vertices. That is a 2-regular connected graph on $n$ vertices. Furthermore, we view $C_{n}$ as a plane graph, that is, the graph together with a planar embedding. For topological use of $C_n$, one may identify it with the regular polygon on $n$ vertices. Vertices of $C_n$ are normally labeled as $v_1, v_2, \ldots, v_n$. The \textit{exact square} of $C_n$, denoted $C^{\#2}_n$, is the graph on the same set of vertices where two vertices are adjacent if they are at a distance (exactly) 2 in $C_n$. Observe that for odd values of $n$, $C^{\#2}_n$ is also a cycle of length $n$. For even values of $n$, $C^{\#2}_n$ consists of two connected components, each isomorphic to a cycle of length $\frac{n}{2}$. They are induced on sets of vertices with odd and even indices and will be denoted, respectively, by $C^{\#2o}_n$ and $C^{\#2e}_n$.

	Given a positive real number, we denote by $O_r$ the (geometric) circle of circumference $r$. That would be a circle of radius $\frac{r}{2\pi}$. The \textit{antipodal} of a point $x$ on $O_r$ is the unique point $\overline{x}$ on $O_r$ which is collinear with $x$ and the center of the circle.
	
	Given a real number $r$, $r\geq 2$, a \textit{circular $r$-coloring} of a signed graph $(G, \sigma)$ is a mapping $\psi$ of the vertices of $G$ to the points of $O_r$ in such a way that when $xy$ is a negative edge, then the distance of $\psi(x)$ from $\psi(y)$ on $O_r$ is at least $1$ and if $xy$ is a positive edge, then the distance of $\psi(x)$ from $\overline{\psi(y)}$ is at least 1, equivalently, the distance between $\psi(x)$ and $\psi(y)$ is at most $\frac{r}{2}-1$. The \textit{circular chromatic number} of $(G, \sigma)$, denoted $\chi_c(G, \sigma)$, is the infimum of $r$ such that $(G, \sigma)$ admits a circular $r$-coloring. When restricted to signed graphs where all edges are negative, we have the classic notion of circular coloring of graphs. This extension to signed graphs is first presented in \cite{NWZ21} noting that a different but similar parameter under a similar name has been introduced in \cite{KS18}. However, compared to \cite{NWZ21}, the role of positive and negative edges are exchanged for better suitability with literature on structural theory on signed graphs, especially in regard to the minor theory of signed graphs.
	
	Among basic results, the following should be noted for the purpose of this work. The infimum in the definition is always attained for finite graphs, even allowing multi-edges and positive loops, but a negative loop cannot be colored with a finite $r$. For the class of signed bipartite (multi)graphs, we have the trivial upper bound of $\chi_c(G, \sigma)\leq 4$, to see this, map the vertices of one part of $G$ to the north pole of $O_4$ and the vertices of the other part to the east point. Even with such a strong upper bound the problem of determining the exact value of the circular chromatic number of a given signed bipartite graph is of high importance and, in general, quite a difficult problem. In particular, as it is pointed out in \cite{NWZ21}, using some basic graph operations, namely indicators, one can transform a graph $G$ into a signed bipartite graph $F(G)$ such that the circular chromatic number of $F(G)$ determines the circular chromatic number of $G$. A basic example of this sort is the construction $S(G)$, which is obtained from a given graph $G$ by replacing each edge $uv$ of $G$ with a negative 4-cycle $ux_{uv}vy_{uv}$ where $x_{uv}$ and $y_{uv}$ are new and distinct vertices. It is then shown in \cite{NWZ21} that $\chi_c(S(G))=4-\frac{4}{\chi_c(G)+1}$. Further connections with some well-known study and theorems, such as the four-color theorem, is discussed in \cite{KNNW21+} and \cite{NW21+}. 
	
	Motivated by these observations and in connection with some other studies, some of which are mentioned in the last section, the question of constructing signed bipartite graphs of high negative girth but circular chromatic number 4 is of high interest. In this work, we present two bipartite analogues of the generalized Mycielski graph on odd cycles as examples of signed bipartite graphs.

	The proofs also lead to an elementary understanding of the relation between coloring problems of graphs and basic notions of algebraic topology, namely the \textit{winding number}.

 Given a closed curve $\gamma$ on the plane, the winding number of $\gamma$, defined rather intuitively, is the number of times $\gamma$ is winded around the origin in the clockwise direction, noting that: if the origin is not in the part bounded by $\gamma$, then the winding number is 0 and that winding in anticlockwise direction is presented by a negative number. Here the closed curves we work with are mappings to $O_r$ with the center of $O_r$ being the center of the plane. They can be thought of as continuous mappings of $[0,1]$ to $O_r$ with the condition that the two endpoints, i.e., $0$ and $1$ are mapped to the same point.  
	
	\section{A historical note}\label{sec:history}
	In 1955 Mycielski introduced the construction \cite{M55} that is now known as the Mycielski construction. His goal of the construction was to build triangle-free graphs of high chromatic number. In this construction, given a graph $G$ one adds a vertex $v'$ for each vertex $v$ of $G$, which is joined to all neighbors of $v$ in $G$ and then adds a vertex $u$ which is joined to all vertices $v'$. It is not difficult to prove that the resulting graph has chromatic number $\chi(G)+1$.
	
	Generalization of the construction, where one adds several layers of copy vertices before adding a universal vertex to the last layer, was first considered independently in Habilitation thesis of M. Stiebitz \cite{S85} and Ph.D. thesis of N. Van Ngoc \cite{V87}. (The former is written in German, but its result can also be found in \cite{GYJS, M03} and the latter is in Hungarian.) Stiebitz applied methods of algebraic topology to prove that if one starts with $K_2$ and iteratively builds a generalized Mycielski, at each step the chromatic number would increase by 1. This does not hold for every graph, though. For example, the chromatic number of the complement of $C_7$ is 4, and any generalized Mycielski of it, except the original one, is also of chromatic number 4. It has been shown recently in \cite{MS19} that the result of Stiebitz is equivalent to the Borsuk-Ulam theorem.
	
	First English publications of the fact that the generalized Mycielski based on an odd cycle has chromatic number 4 appeared independently in \cite{P92, VT95,Y96}. The proof of Payan \cite{P92} is about the special case of $M_{k}(C_{2k+1})$ as they appear as subgraphs of nonbipartite Cayley graphs on binary groups, but it works the same for any $M_{\ell}(C_{2k+1})$. This proof has strongly motivated the work presented here. The proof of \cite{VT95} is presented quite differently, but the hidden idea behind the proof is the same. The result of \cite{Y96} is more general. It is shown that if $G$ is not bipartite but admits an embedding on the projective plane where all facial cycles are 4-cycles, then $\chi (G)=4$. That such structures are necessary for 4-chromatic triangle-free projective planar graphs was conjectured in \cite{Y96} and proved in \cite{GT97}. The well-known fact that $M_{\ell}(C_{2k+1})$ quadrangulate the projective plane is evident from our presentation of these graphs in the next section.

	The circular chromatic number of Mycielski constructions was first studied in \cite{CHZ99}. That of the generalized Mycielski is studied in  \cite{DGMVZ05,HZ03,LLGS03,ST06} among others. In particular, that $\chi_c(M_{\ell}(C_{2k+1}))=4$ follows, independently, from the general results of \cite{DGMVZ05} and of \cite{ST06}. In the latter, it is shown that if the lower bound of $2k$ for the chromatic number is proved using topological connectivity, then the same lower bound works for the circular chromatic number as well. 

	\section{The construction}\label{sec:Construction}

	 The main body of the construction we will work with is an almost quadrangulation of the cylinder which we define here. Given positive integers $\ell$ and $k$, $C_{_{\ell\times k}}$ is the graph whose vertex set is $V=\{v_{_{i,j}} \mid 1 \leq i \leq \ell, 1\leq j \leq k\}$ with the edge set $E=\{ v_{_{i,j}}v_{_{i+1, j-1}},  v_{_{i,j}}v_{_{i+1, j}} \mid 1 \leq i \leq \ell-1, 1\leq j \leq k\}$. Here, and in the rest of this work, the addition on the indices is taken modular the maximum value of the said index, which is $\pmod {k}$ in this case. For a geometric presentation, as depicted in Figure~\ref{fig:ell-Times-2k+1}, we consider $\ell$ parallel and non-contractible circles on the cylinder labeled $1, 2, \dots, \ell$. Then for odd values of $\ell$ we take points $\frac{(2j-1)\pi}{k}$ ($j=1,2, \dots, k$) on them and for even values of $\ell$ we take points $\frac{(2j)\pi}{k}$ ($j=1,2, \dots, k$) to present the vertices of $C_{_{\ell\times k}}$.

	We note that, as a graph $C_{_{\ell\times (2k+1)}}$ is isomorphic to the categorical product $P_{\ell} \times C_{2k+1}$. A general picture of this graph is depicted in Figure~\ref{fig:ell-Times-2k+1} where the dashed circles are only presenting the layers, but they will play a key role.

	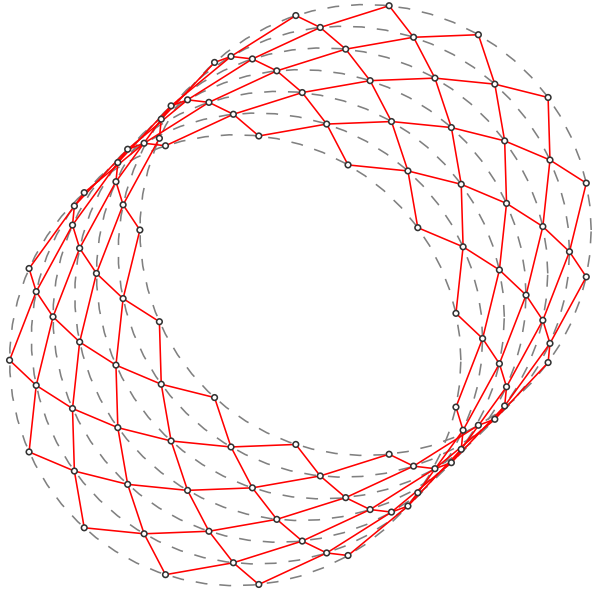
\begin{figure}[ht]
		\centering
		\scalebox{1.5}{
			\tikzset{math3d/.style= {x={(2cm,0cm)}, y={(0cm,2cm) }}}
			\begin{tikzpicture}[math3d]
				\newcommand{\n}{15}
				\newcommand{\h}{7}
				\newcommand{\ch}{12}
				\newcommand{\rl}{1}
				\newcommand{\rh}{1}
				
				\foreach \i in {1,...,\h}{
					\path[draw, dashed, gray, opacity=.2] plot[domain=0:2*pi,samples=4*\n] ({\rh*cos(\x r)}, {\rh*sin(\x r)}, .5*\i);
					
				}

				\foreach \i in {1,3,5}{
					\foreach \t in {1,...,\n} {
						
						\draw [ thin, red, opacity=.6] ( ({(\rl)*cos((2*\t+1)*pi/\n r)},{\rl*sin((2*\t+1)*pi/\n r)},.5*\i) -- ({\rh*cos((2*\t+1)*pi/\n r-\ch)},{\rh*sin((2*\t+1)*pi/\n r-\ch)},.5*\i+.5)  -- cycle;
						
						\draw [ thin, red, opacity=.6] ( ({(\rl)*cos((2*\t+1)*pi/\n r)},{\rl*sin((2*\t+1)*pi/\n r)},.5*\i) -- ({\rh*cos((2*\t+1)*pi/\n r+\ch)},{\rh*sin((2*\t+1)*pi/\n r+\ch)},.5*\i+.5)  -- cycle;
					}

				}
				
				\foreach \i in {2,4,6}{
					\foreach \t in {1,...,\n} {
						
						\draw [ thin, red, opacity=.6] ( ({(\rl)*cos((2*\t+1)*pi/\n r+\ch)},{\rl*sin((2*\t+1)*pi/\n r+\ch)},.5*\i) -- ({\rh*cos((2*\t+1)*pi/\n r)},{\rh*sin((2*\t+1)*pi/\n r)},.5*\i+.5)  -- cycle;
						
						\draw [ thin, red, opacity=.6] ( ({(\rl)*cos((2*\t+1)*pi/\n r-\ch)},{\rl*sin((2*\t+1)*pi/\n r-\ch)},.5*\i) -- ({\rh*cos((2*\t+1)*pi/\n r)},{\rh*sin((2*\t+1)*pi/\n r)},.5*\i+.5)  -- cycle;
					}
				} 
				
				\foreach \i in {1,3,5,7}{

					\foreach \t in {1,...,\n} {
						\node[circle, black, inner sep=0mm, draw=black!80, fill=white, minimum size=.5mm] at  ({(\rl)*cos((2*\t+1)*pi/\n r)},{\rl*sin((2*\t+1)*pi/\n r)},.5*\i){};
					}
				}
				
				\foreach \i in {2,4,6}{

					\foreach \t in {1,...,\n} {
						\node[circle, red, inner sep=0mm, draw=black!80, fill=white, minimum size=.5mm] at  ({(\rl)*cos((2*\t)*pi/\n r)},{\rl*sin((2*\t)*pi/\n r)},.5*\i){};
					}
				}

			\end{tikzpicture}
		}
		%
			%
				%
				%
			%
			%
			%
		%
				%
				%
		%
		%
		\caption{$C_{_{\ell\times (2k+1)}}$ with layers highlighted.}
		\label{fig:ell-Times-2k+1}
	\end{figure}

	\subsection{$M_{\ell}(C_{2k+1})$}
	Given positive integers $\ell$ and $k$, the generalized Mycielski graph of the odd cycle $C_{2k+1}$, $M_{\ell}(C_{2k+1})$ is built from $C_{_{\ell\times (2k+1)}}$ by the following two steps: 
	
	\begin{itemize}
		\item Connect $v_{_{1,j}}$ to $v_{_{1,j+k}}$ (Figure~\ref{fig:CylinderPlusCycle}, right).  
		\item Add a new vertex $u$ and connect it to all vertices $v_{_{\ell,j}}$, $j=1,\ldots, 2k+1$ (Figure~\ref{fig:CylinderPlusCycle}, left).
	\end{itemize}

	\begin{figure}[ht]
		\centering
		
		\begin{minipage}{.4\textwidth}

			\scalebox{1.3}{
				\tikzset{math3d/.style= {x={(2cm,0cm)}, y={(0cm,2cm)}}}
				\begin{tikzpicture}[math3d]
					\newcommand{\n}{15}
					\newcommand{\h}{7}
					\newcommand{\ch}{12}
					\newcommand{\rl}{1}
					\newcommand{\rh}{1}
					
					
					\path[draw, gray] plot[domain=0:2*pi,samples=4*\n] ({\rh*cos(\x r)}, {\rh*sin(\x r)}, 3.5);
					
					\foreach \i in {1,...,6}{
						\path[draw, dashed, gray] plot[domain=0:2*pi,samples=4*\n] ({\rh*cos(\x r)}, {\rh*sin(\x r)}, .5*\i);
					}
					
					%
					
					\foreach \i in {1,3,5}{
						\foreach \t in {1,...,\n} {
							
							\draw [thick, red, opacity=0.2] ( ({(\rl)*cos((2*\t+1)*pi/\n r)},{\rl*sin((2*\t+1)*pi/\n r)},.5*\i) -- ({\rh*cos((2*\t+1)*pi/\n r-\ch)},{\rh*sin((2*\t+1)*pi/\n r-\ch)},.5*\i+.5)  -- cycle;
							
							\draw [ thick, red, opacity=0.2] ( ({(\rl)*cos((2*\t+1)*pi/\n r)},{\rl*sin((2*\t+1)*pi/\n r)},.5*\i) -- ({\rh*cos((2*\t+1)*pi/\n r+\ch)},{\rh*sin((2*\t+1)*pi/\n r+\ch)},.5*\i+.5)  -- cycle;
						}

					}
					
					\foreach \i in {2,4,6}{
						\foreach \t in {1,...,\n} {
							
							\draw [ thick, red, opacity=0.2] ( ({(\rl)*cos((2*\t+1)*pi/\n r+\ch)},{\rl*sin((2*\t+1)*pi/\n r+\ch)},.5*\i) -- ({\rh*cos((2*\t+1)*pi/\n r)},{\rh*sin((2*\t+1)*pi/\n r)},.5*\i+.5)  -- cycle;
							
							\draw [ thick, red, opacity=0.2] ( ({(\rl)*cos((2*\t+1)*pi/\n r-\ch)},{\rl*sin((2*\t+1)*pi/\n r-\ch)},.5*\i) -- ({\rh*cos((2*\t+1)*pi/\n r)},{\rh*sin((2*\t+1)*pi/\n r)},.5*\i+.5)  -- cycle;
						}
					}

					\foreach \t in {1,...,\n} {
						
						\draw [  red, opacity=1] ( ({(cos((2*\t+14)*pi/\n r+\ch)},{sin((2*\t+14)*pi/\n r+\ch)},3.5) -- (-.5,-.5)  -- cycle;
						
					}
					\draw node[circle, red, inner sep=0mm, draw=black!80, fill=white, minimum size=1mm] at (-.5,-.5) {};
				\end{tikzpicture}
			}
		\end{minipage}
		\begin{minipage}{.4\textwidth}
			\scalebox{1.3}{
				\tikzset{math3d/.style= {x={(2cm,0cm)}, y={(0cm,2cm)}}}
				\begin{tikzpicture}[math3d]
					\newcommand{\n}{15}
					\newcommand{\h}{7}
					\newcommand{\ch}{12}
					\newcommand{\rl}{1}
					\newcommand{\rh}{1}
					
					\path[draw, gray] plot[domain=0:2*pi,samples=4*\n, opacity=.5] ({\rh*cos(\x r)}, {\rh*sin(\x r)}, .5);
					
					\foreach \i in {1,...,\h}{
						\path[draw, dashed, gray] plot[domain=0:2*pi,samples=4*\n] ({\rh*cos(\x r)}, {\rh*sin(\x r)}, .5*\i);
						
						%
						
					}
					\foreach \i in {1,3,5}{
						\foreach \t in {1,...,\n} {
							
							\draw [thick, red, opacity=0.2] ( ({(\rl)*cos((2*\t+1)*pi/\n r)},{\rl*sin((2*\t+1)*pi/\n r)},.5*\i) -- ({\rh*cos((2*\t+1)*pi/\n r-\ch)},{\rh*sin((2*\t+1)*pi/\n r-\ch)},.5*\i+.5)  -- cycle;
							
							\draw [ thick, red, opacity=0.2] ( ({(\rl)*cos((2*\t+1)*pi/\n r)},{\rl*sin((2*\t+1)*pi/\n r)},.5*\i) -- ({\rh*cos((2*\t+1)*pi/\n r+\ch)},{\rh*sin((2*\t+1)*pi/\n r+\ch)},.5*\i+.5)  -- cycle;
						}

					}
					
					\foreach \i in {2,4,6}{
						\foreach \t in {1,...,\n} {
							
							\draw [ thick, red, opacity=0.2] ( ({(\rl)*cos((2*\t+1)*pi/\n r+\ch)},{\rl*sin((2*\t+1)*pi/\n r+\ch)},.5*\i) -- ({\rh*cos((2*\t+1)*pi/\n r)},{\rh*sin((2*\t+1)*pi/\n r)},.5*\i+.5)  -- cycle;
							
							\draw [ thick, red, opacity=0.2] ( ({(\rl)*cos((2*\t+1)*pi/\n r-\ch)},{\rl*sin((2*\t+1)*pi/\n r-\ch)},.5*\i) -- ({\rh*cos((2*\t+1)*pi/\n r)},{\rh*sin((2*\t+1)*pi/\n r)},.5*\i+.5)  -- cycle;
						}
					}

					\foreach \t in {1,...,\n} {
						
						\draw [ red, opacity=1] ( ({(cos((2*\t+14)*pi/\n r+\ch)},{sin((2*\t+14)*pi/\n r+\ch)},.5) -- ({cos((2*\t+1)*pi/\n r)},{sin((2*\t+1)*pi/\n r)},.5)  -- cycle;
						
					}
					
				\end{tikzpicture}
			}
		\end{minipage}
		\caption{Constructions on bottom and top layers.}
		\label{fig:CylinderPlusCycle}
	\end{figure}
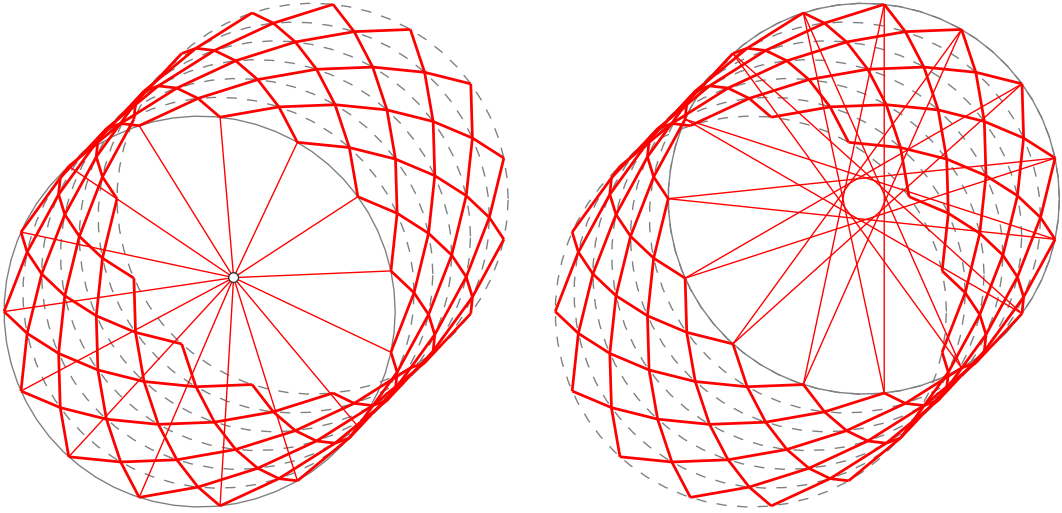
	
	Observe that the added edges in the first item form an isomorphic copy of $C_{2k+1}$. One can easily observe that starting with this cycle, the classic definition of a generalized Mycielski graph results in the same graph.
	The graph $M_1(C_3)$ is $K_4$. The graph $M_2(C_5)$ is the well-known Gr\"{o}zsch graph. 
	To show that $M_2(C_5)$ is the smallest 4-chromatic triangle-free graph is proposed as an exercise in \cite{H69}. 
	Furthermore, Chvátal showed in \cite{C74} that $M_2(C_5)$ is the only 4-chromatic triangle-free graph on 11 vertices. 
	
	The following is a key property of $M_{\ell}(C_{2k+1})$. 
	
	\begin{proposition}
		The length of the shortest odd cycle of $M_{\ell}(C_{2k+1})$ is the $\min\{2k+1,2l+1\}$.
	\end{proposition}
	
	Since this is a folklore fact, we do not provide a proof but we note that the main idea to verify it is also presented in the next proposition.
	
	\subsection{$\widehat{BQ}(\ell,2k+1)$}

	Next, given integers $\ell$ and $k$ satisfying $\ell \geq 2$ and $k \geq 1$, we define the signed bipartite graph  $\widehat{BQ}(\ell,2k+1)$ also from $C_{\ell \times (2k+1)}$ as follows.   
	
	\begin{itemize}
		\item Edges of $C_{\ell \times (2k+1)}$ are all negative.
		\item Connect $v_{_{1,j}}$ to $v_{_{2,j+k}}$ by a positive edge (Figure~\ref{fig:CylinderPlusMobuis}, right).  
		\item Add a new vertex $u$ and connect it to each of the vertices $v_{_{\ell,j}}$, $j=1,\ldots, 2k+1$, with a negative edge (Figure~\ref{fig:CylinderPlusMobuis}, left).
	\end{itemize}

	\begin{figure}[h]
		\centering
		
		\begin{minipage}{.4\textwidth}

			\scalebox{1.3}{
				\tikzset{math3d/.style= {x={(2cm,0cm)}, y={(0cm,2cm)}}}
				\begin{tikzpicture}[math3d]
					\newcommand{\n}{15}
					\newcommand{\h}{7}
					\newcommand{\ch}{12}
					\newcommand{\rl}{1}
					\newcommand{\rh}{1}
					
					\path[draw, gray, line width=0.05mm] plot[domain=0:2*pi,samples=4*\n] ({\rh*cos(\x r)}, {\rh*sin(\x r)}, 3.5);
					
					\foreach \i in{1,2}{
						\path[draw, gray] plot[domain=0:2*pi,samples=4*\n] ({\rh*cos(\x r)}, {\rh*sin(\x r)}, .5*\i);
					}
					
					\foreach \i in {3,...,6}{
						\path[draw, dashed, gray] plot[domain=0:2*pi,samples=4*\n] ({\rh*cos(\x r)}, {\rh*sin(\x r)}, .5*\i);

					}
					\foreach \i in {1,3,5}{
						\foreach \t in {1,...,\n} {
							
							\draw [thick, red, line width=0.15mm, opacity=0.2] ( ({(\rl)*cos((2*\t+1)*pi/\n r)},{\rl*sin((2*\t+1)*pi/\n r)},.5*\i) -- ({\rh*cos((2*\t+1)*pi/\n r-\ch)},{\rh*sin((2*\t+1)*pi/\n r-\ch)},.5*\i+.5)  -- cycle;
							
							\draw [ thick, red, line width=0.15mm, opacity=0.2] ( ({(\rl)*cos((2*\t+1)*pi/\n r)},{\rl*sin((2*\t+1)*pi/\n r)},.5*\i) -- ({\rh*cos((2*\t+1)*pi/\n r+\ch)},{\rh*sin((2*\t+1)*pi/\n r+\ch)},.5*\i+.5)  -- cycle;
						}

					}
					
					\foreach \i in {2,4,6}{
						\foreach \t in {1,...,\n} {
							
							\draw [ thick, red, line width=0.15mm, opacity=0.2] ( ({(\rl)*cos((2*\t+1)*pi/\n r+\ch)},{\rl*sin((2*\t+1)*pi/\n r+\ch)},.5*\i) -- ({\rh*cos((2*\t+1)*pi/\n r)},{\rh*sin((2*\t+1)*pi/\n r)},.5*\i+.5)  -- cycle;
							
							\draw [ thick, red, line width=0.15mm, opacity=0.2] ( ({(\rl)*cos((2*\t+1)*pi/\n r-\ch)},{\rl*sin((2*\t+1)*pi/\n r-\ch)},.5*\i) -- ({\rh*cos((2*\t+1)*pi/\n r)},{\rh*sin((2*\t+1)*pi/\n r)},.5*\i+.5)  -- cycle;
						}
					}

					\foreach \t in {1,...,\n} {
						
						\draw [ thick, red, line width=0.15mm, opacity=0.9] ( ({(cos((2*\t+14)*pi/\n r+\ch)},{sin((2*\t+14)*pi/\n r+\ch)},3.5) -- (-.5,-.5)  -- cycle;
						
					}
					\draw node[line width=0.05mm, circle, fill=white, minimum width=.5mm, inner sep=0mm, draw=black!80, minimum size=1mm] at (-.5,-.5) {};
					
						%
						%
					
				\end{tikzpicture}
			}
		\end{minipage}
		\begin{minipage}{.4\textwidth}
			\scalebox{1.3}{
				\tikzset{math3d/.style= {x={(2cm,0cm)}, y={(0cm,2cm)}}}
				\begin{tikzpicture}[math3d]
					\newcommand{\n}{15}
					\newcommand{\h}{7}
					\newcommand{\ch}{12}
					\newcommand{\rl}{1}
					\newcommand{\rh}{1}
					
					\foreach \i in{1,2}{
						\path[draw, gray] plot[domain=0:2*pi,samples=4*\n] ({\rh*cos(\x r)}, {\rh*sin(\x r)}, .5*\i);
					}
					
					\foreach \i in {3,...,\h}{
						\path[draw, dashed, gray] plot[domain=0:2*pi,samples=4*\n] ({\rh*cos(\x r)}, {\rh*sin(\x r)}, .5*\i);
						
					}
					\foreach \i in {1,3,5}{
						\foreach \t in {1,...,\n} {
							
							\draw [thick, red, opacity=0.2] ( ({(\rl)*cos((2*\t+1)*pi/\n r)},{\rl*sin((2*\t+1)*pi/\n r)},.5*\i) -- ({\rh*cos((2*\t+1)*pi/\n r-\ch)},{\rh*sin((2*\t+1)*pi/\n r-\ch)},.5*\i+.5)  -- cycle;
							
							\draw [ thick, red, opacity=0.2] ( ({(\rl)*cos((2*\t+1)*pi/\n r)},{\rl*sin((2*\t+1)*pi/\n r)},.5*\i) -- ({\rh*cos((2*\t+1)*pi/\n r+\ch)},{\rh*sin((2*\t+1)*pi/\n r+\ch)},.5*\i+.5)  -- cycle;
						}

					}
					
					\foreach \i in {2,4,6}{
						\foreach \t in {1,...,\n} {
							
							\draw [ thick, red, opacity=0.2] ( ({(\rl)*cos((2*\t+1)*pi/\n r+\ch)},{\rl*sin((2*\t+1)*pi/\n r+\ch)},.5*\i) -- ({\rh*cos((2*\t+1)*pi/\n r)},{\rh*sin((2*\t+1)*pi/\n r)},.5*\i+.5)  -- cycle;
							
							\draw [ thick, red, opacity=0.2] ( ({(\rl)*cos((2*\t+1)*pi/\n r-\ch)},{\rl*sin((2*\t+1)*pi/\n r-\ch)},.5*\i) -- ({\rh*cos((2*\t+1)*pi/\n r)},{\rh*sin((2*\t+1)*pi/\n r)},.5*\i+.5)  -- cycle;
						}
					}

					\foreach \t in {1,...,\n} {
						
						\draw [line width=0.075mm,  blue, opacity=0.9] ( ({(cos((2*\t+14)*pi/\n r+\ch)},{sin((2*\t+14)*pi/\n r+\ch)},.5) .. controls ({cos((2*\t+14)*pi/\n r+\ch)+cos((2*\t)*pi/\n r)}, {sin((2*\t+14)*pi/\n r+\ch)+sin((2*\t)*pi/\n r)})  .. ({cos((2*\t)*pi/\n r)},{sin((2*\t)*pi/\n r)},1);
						
					}

				\end{tikzpicture}
			}
		\end{minipage}
		\caption{Construction of $\widehat{BQ}(\ell,2k+1)$.}
		\label{fig:CylinderPlusMobuis}
	\end{figure}
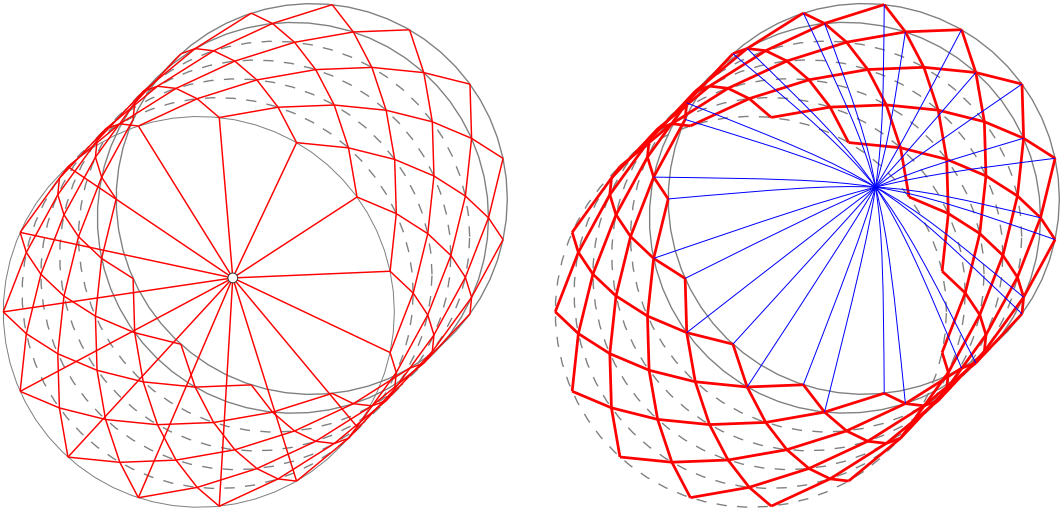

	We view this construction as one of the bipartite analogues of the generalized Mycielski. The second item of the construction, which is presented in Figure~\ref{fig:CylinderPlusMobuis} (right) is the main difference with the previously known constructions: While in construction of $M_{\ell}(C_{2k+1})$ we add some edges between vertices of the first layer, in this new construction we add some connection between vertices of the first layer and the second layer. Therefore this operation preserves the bipartition. The underlying graph of the induced subgraph on the first two layers is isomorphic to what is known as the M\"obuis ladder with $2k+1$ steps. We will refer to it as such.

	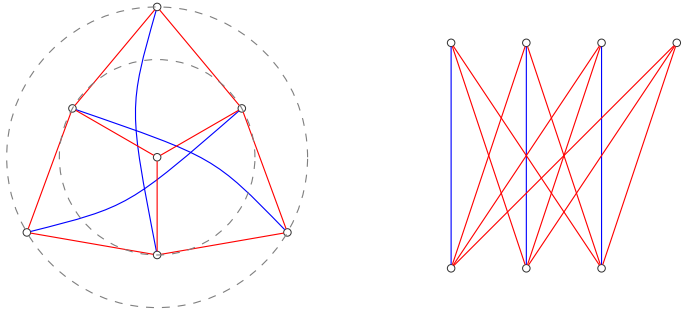
\begin{figure}[h]
		\centering
		\begin{minipage}{.4\textwidth}
			\centering
			\scalebox{1}{
				\begin{tikzpicture}
					
					\foreach \i in {1,2,3}
					{ \draw node[line width=0.05mm, circle, fill=white, minimum width=.5mm, inner sep=0mm, draw=black!80, minimum size=1mm]  (\i) at ({\i*120+30}:1.3){};}
					
					\foreach \i in {4,5,6}
					{ \draw node[line width=0.05mm, circle, fill=white, minimum width=.5mm, inner sep=0mm, draw=black!80, minimum size=1mm]  (\i) at ({\i*120+90}:2){};
					}
					
					\draw [ thick, red, line width=0.15mm, opacity=0.9] (1) -- (4);
					\draw [ thick, red, line width=0.15mm, opacity=0.9] (2) -- (5);
					\draw [ thick, red, line width=0.15mm, opacity=0.9] (3) -- (6);
					
					\draw [ thick, red, line width=0.15mm, opacity=0.9] (1) -- (6);
					\draw [ thick, red, line width=0.15mm, opacity=0.9] (2) -- (4);
					\draw [ thick, red, line width=0.15mm, opacity=0.9] (3) -- (5);
					
					\draw [ thick, blue, line width=0.15mm, opacity=0.9] (1) .. controls({0}:.76) .. (5);
					\draw [ thick, blue, line width=0.15mm, opacity=0.9] (2) .. controls({120}:.76) .. (6);
					\draw [ thick, blue, line width=0.15mm, opacity=0.9] (3) .. controls({240}:.76) .. (4);

					\draw  [ dashed, gray, line width=0.15mm, opacity=0.2] (0,0) circle (1.3cm);
					\draw  [ dashed, gray, line width=0.15mm, opacity=0.2] (0,0) circle (2cm);
					
					\draw node[line width=0.05mm, circle, fill=white, minimum width=.5mm, inner sep=0mm, draw=black!80, minimum size=1mm]  (7) at (0,0){};
					
					\draw [ thick, red, line width=0.15mm, opacity=0.9] (7) -- (1);
					\draw [ thick, red, line width=0.15mm, opacity=0.9] (7) -- (2);
					\draw [ thick, red, line width=0.15mm, opacity=0.9] (7) -- (3);
					
				\end{tikzpicture}
			}
		\end{minipage}
		\centering
		\begin{minipage}{.4\textwidth}
			\scalebox{1}{
				\begin{tikzpicture}
					
					\foreach \i in {1,2,3}
					{ \draw node[line width=0.05mm, circle, fill=white, minimum width=.5mm, inner sep=0mm, draw=black!80, minimum size=1mm]  (\i) at (\i, 0){};}
					
					\foreach \i in {4,5,6}
					{ \draw node[line width=0.05mm, circle, fill=white, minimum width=.5mm, inner sep=0mm, draw=black!80, minimum size=1mm]  (\i) at (\i-3,3){};
					}
					
					\draw [ thick, blue, line width=0.15mm, opacity=0.9] (1) -- (4);
					\draw [ thick, blue, line width=0.15mm, opacity=0.9] (2) -- (5);
					\draw [ thick, blue, line width=0.15mm, opacity=0.9] (3) -- (6);
					
					\draw [ thick, red, line width=0.15mm, opacity=0.9] (1) -- (6);
					\draw [ thick, red, line width=0.15mm, opacity=0.9] (2) -- (4);
					\draw [ thick, red, line width=0.15mm, opacity=0.9] (3) -- (5);
					
					\draw [ thick, red, line width=0.15mm, opacity=0.9] (1) -- (5);
					\draw [ thick, red, line width=0.15mm, opacity=0.9] (2) -- (6);
					\draw [ thick, red, line width=0.15mm, opacity=0.9] (3) -- (4);

					\draw node[line width=0.05mm, circle, fill=white, minimum width=.5mm, inner sep=0mm, draw=black!80, minimum size=1mm]  (7) at (4,3){};
					
					\draw [ thick, red, line width=0.15mm, opacity=0.9] (7) -- (1);
					\draw [ thick, red, line width=0.15mm, opacity=0.9] (7) -- (2);
					\draw [ thick, red, line width=0.15mm, opacity=0.9] (7) -- (3);
					
				\end{tikzpicture}
			}
		\end{minipage}
		
		\caption{$\widehat{BQ}(2,3)$, presented two different ways.}
		\label{fig:BM2,3}
	\end{figure}

	The case of $\widehat{BQ}(2,3)$ is $(K_{3,4},M)$ depicted in Figure~\ref{fig:BM2,3}. It is the signed bipartite graph where each one of the edges of a maximum matching of $K_{3,4}$ is assigned a positive sign and all the other edges are assigned a negative sign.

	\begin{figure}[h]
		\centering
		\begin{minipage}[t]{.4\textwidth}
			\centering
			\scalebox{1}{
				\begin{tikzpicture}

					\foreach \i in {1,2,3,4,5}
					{ \draw node[line width=0.05mm, circle, fill=white, minimum width=.5mm, inner sep=0mm, draw=black!80, minimum size=1mm]  (\i) at ({\i*72}:1.2){};}
					
					\foreach \i in {6,7,8,9,10}
					{ \draw node[line width=0.05mm, circle, fill=white, minimum width=.5mm, inner sep=0mm, draw=black!80, minimum size=1mm]  (\i) at ({\i*72+36}:2){};
					}
					
					\foreach \i in {11,12,13,14,15}
					{ \draw node[line width=0.05mm, circle, fill=white, minimum width=.5mm, inner sep=0mm, draw=black!80, minimum size=1mm]  (\i) at ({\i*72}:2.8){};}

					\draw [ thick, red, line width=0.15mm, opacity=0.9] (1) -- (6);
					\draw [ thick, red, line width=0.15mm, opacity=0.9] (2) -- (7);
					\draw [ thick, red, line width=0.15mm, opacity=0.9] (3) -- (8);
					\draw [ thick, red, line width=0.15mm, opacity=0.9] (4) -- (9);
					\draw [ thick, red, line width=0.15mm, opacity=0.9] (5) -- (10);
					
					\draw [ thick, red, line width=0.15mm, opacity=0.9] (1) -- (10);
					\draw [ thick, red, line width=0.15mm, opacity=0.9] (2) -- (6);
					\draw [ thick, red, line width=0.15mm, opacity=0.9] (3) -- (7);
					\draw [ thick, red, line width=0.15mm, opacity=0.9] (4) -- (8);
					\draw [ thick, red, line width=0.15mm, opacity=0.9] (5) -- (9);
					
					\draw [ thick, red, line width=0.15mm, opacity=0.9] (6) -- (11);
					\draw [ thick, red, line width=0.15mm, opacity=0.9] (7) -- (12);
					\draw [ thick, red, line width=0.15mm, opacity=0.9] (8) -- (13);
					\draw [ thick, red, line width=0.15mm, opacity=0.9] (9) -- (14);
					\draw [ thick, red, line width=0.15mm, opacity=0.9] (10) -- (15);
					
					\draw [ thick, red, line width=0.15mm, opacity=0.9] (6) -- (12);
					\draw [ thick, red, line width=0.15mm, opacity=0.9] (7) -- (13);
					\draw [ thick, red, line width=0.15mm, opacity=0.9] (8) -- (14);
					\draw [ thick, red, line width=0.15mm, opacity=0.9] (9) -- (15);
					\draw [ thick, red, line width=0.15mm, opacity=0.9] (10) -- (11);

					\draw  [ dashed, gray, line width=0.15mm, opacity=0.2] (0,0) circle (1.2cm);
					\draw  [ dashed, gray, line width=0.15mm, opacity=0.2] (0,0) circle (2cm);
					\draw  [ dashed, gray, line width=0.15mm, opacity=0.2] (0,0) circle (2.8cm);
					
					\draw node[line width=0.05mm, circle, fill=white, minimum width=.5mm, inner sep=0mm, draw=black!80, minimum size=1mm]  (16) at ({90}:0){};
					
					\draw [ thick, red, line width=0.15mm, opacity=0.9] (1) -- (16);
					\draw [ thick, red, line width=0.15mm, opacity=0.9] (2) -- (16);
					\draw [ thick, red, line width=0.15mm, opacity=0.9] (3) -- (16);
					\draw [ thick, red, line width=0.15mm, opacity=0.9] (4) -- (16);
					\draw [ thick, red, line width=0.15mm, opacity=0.9] (5) -- (16);
					
					\draw [ thick, blue, line width=0.15mm, opacity=0.9] (6) .. controls({0}:.29) .. (14);
					\draw [ thick, blue, line width=0.15mm, opacity=0.9] (7) .. controls({72}:.29) .. (15);
					\draw [ thick, blue, line width=0.15mm, opacity=0.9] (8) .. controls({144}:.29) .. (11);
					\draw [ thick, blue, line width=0.15mm, opacity=0.9] (9) .. controls({216}:.29) .. (12);
					\draw [ thick, blue, line width=0.15mm, opacity=0.9] (10) .. controls({288}:.29) .. (13);
					
				\end{tikzpicture}
			}
			\caption{$\widehat{BQ}(3,5)$.}
			\label{fig:BM3,5}
		\end{minipage}
		\centering
		\begin{minipage}[t]{.4\textwidth}
			\centering
			\scalebox{1}{
				\begin{tikzpicture}
					
					\foreach \i in {1,2,3,4,5,6,7}
					{ \draw node[line width=0.05mm, circle, fill=white, minimum width=.5mm, inner sep=0mm, draw=black!80, minimum size=1mm]  (\i) at ({\i*360/7}:1.2){};}
					
					\foreach \i in {8,9,10,11,12,13,14}
					{ \draw node[line width=0.05mm, circle, fill=white, minimum width=.5mm, inner sep=0mm, draw=black!80, minimum size=1mm]  (\i) at ({\i*360/7+180/7}:2){};
					}
					
					\foreach \i in {15,16,17,18,19,20,21}
					{ \draw node[line width=0.05mm, circle, fill=white, minimum width=.5mm, inner sep=0mm, draw=black!80, minimum size=1mm]  (\i) at ({\i*360/7}:2.8){};}
					
					\foreach \i in {22,23,24,25,26,27,28}
					{ \draw node[line width=0.05mm, circle, fill=white, minimum width=.5mm, inner sep=0mm, draw=black!80, minimum size=1mm]  (\i) at ({\i*360/7+180/7}:3.6){};}

					\draw [ thick, red, line width=0.15mm, opacity=0.9] (1) -- (8);
					\draw [ thick, red, line width=0.15mm, opacity=0.9] (2) -- (9);
					\draw [ thick, red, line width=0.15mm, opacity=0.9] (3) -- (10);
					\draw [ thick, red, line width=0.15mm, opacity=0.9] (4) -- (11);
					\draw [ thick, red, line width=0.15mm, opacity=0.9] (5) -- (12);
					\draw [ thick, red, line width=0.15mm, opacity=0.9] (6) -- (13);
					\draw [ thick, red, line width=0.15mm, opacity=0.9] (7) -- (14);
					
					\draw [ thick, red, line width=0.15mm, opacity=0.9] (1) -- (14);
					\draw [ thick, red, line width=0.15mm, opacity=0.9] (2) -- (8);
					\draw [ thick, red, line width=0.15mm, opacity=0.9] (3) -- (9);
					\draw [ thick, red, line width=0.15mm, opacity=0.9] (4) -- (10);
					\draw [ thick, red, line width=0.15mm, opacity=0.9] (5) -- (11);
					\draw [ thick, red, line width=0.15mm, opacity=0.9] (6) -- (12);
					\draw [ thick, red, line width=0.15mm, opacity=0.9] (7) -- (13);

					\draw [ thick, red, line width=0.15mm, opacity=0.9] (8) -- (15);
					\draw [ thick, red, line width=0.15mm, opacity=0.9] (9) -- (16);
					\draw [ thick, red, line width=0.15mm, opacity=0.9] (10) -- (17);
					\draw [ thick, red, line width=0.15mm, opacity=0.9] (11) -- (18);
					\draw [ thick, red, line width=0.15mm, opacity=0.9] (12) -- (19);
					\draw [ thick, red, line width=0.15mm, opacity=0.9] (13) -- (20);
					\draw [ thick, red, line width=0.15mm, opacity=0.9] (14) -- (21);

					\draw [ thick, red, line width=0.15mm, opacity=0.9] (8) -- (16);
					\draw [ thick, red, line width=0.15mm, opacity=0.9] (9) -- (17);
					\draw [ thick, red, line width=0.15mm, opacity=0.9] (10) -- (18);
					\draw [ thick, red, line width=0.15mm, opacity=0.9] (11) -- (19);
					\draw [ thick, red, line width=0.15mm, opacity=0.9] (12) -- (20);
					\draw [ thick, red, line width=0.15mm, opacity=0.9] (13) -- (21);
					\draw [ thick, red, line width=0.15mm, opacity=0.9] (14) -- (15);

					\draw [ thick, red, line width=0.15mm, opacity=0.9] (15) -- (22);
					\draw [ thick, red, line width=0.15mm, opacity=0.9] (16) -- (23);
					\draw [ thick, red, line width=0.15mm, opacity=0.9] (17) -- (24);
					\draw [ thick, red, line width=0.15mm, opacity=0.9] (18) -- (25);
					\draw [ thick, red, line width=0.15mm, opacity=0.9] (19) -- (26);
					\draw [ thick, red, line width=0.15mm, opacity=0.9] (20) -- (27);
					\draw [ thick, red, line width=0.15mm, opacity=0.9] (21) -- (28);
					
					\draw [ thick, red, line width=0.15mm, opacity=0.9] (15) -- (28);
					\draw [ thick, red, line width=0.15mm, opacity=0.9] (16) -- (22);
					\draw [ thick, red, line width=0.15mm, opacity=0.9] (17) -- (23);
					\draw [ thick, red, line width=0.15mm, opacity=0.9] (18) -- (24);
					\draw [ thick, red, line width=0.15mm, opacity=0.9] (19) -- (25);
					\draw [ thick, red, line width=0.15mm, opacity=0.9] (20) -- (26);
					\draw [ thick, red, line width=0.15mm, opacity=0.9] (21) -- (27);

					\draw  [ dashed, gray, line width=0.15mm, opacity=0.2] (0,0) circle (1.2cm);
					\draw  [ dashed, gray, line width=0.15mm, opacity=0.2] (0,0) circle (2cm);
					\draw  [ dashed, gray, line width=0.15mm, opacity=0.2] (0,0) circle (2.8cm);
					\draw  [ dashed, gray, line width=0.15mm, opacity=0.2] (0,0) circle (3.6cm);
					
					\draw node[line width=0.05mm, circle, fill=white, minimum width=.5mm, inner sep=0mm, draw=black!80, minimum size=1mm]  (29) at ({90}:0){};
					
					\draw [ thick, red, line width=0.15mm, opacity=0.9] (1) -- (29);
					\draw [ thick, red, line width=0.15mm, opacity=0.9] (2) -- (29);
					\draw [ thick, red, line width=0.15mm, opacity=0.9] (3) -- (29);
					\draw [ thick, red, line width=0.15mm, opacity=0.9] (4) -- (29);
					\draw [ thick, red, line width=0.15mm, opacity=0.9] (5) -- (29);
					\draw [ thick, red, line width=0.15mm, opacity=0.9] (6) -- (29);
					\draw [ thick, red, line width=0.15mm, opacity=0.9] (7) -- (29);
					
					\draw [ thick, blue, line width=0.15mm, opacity=0.9] (15) .. controls({0}:.79) .. (25);
					\draw [ thick, blue, line width=0.15mm, opacity=0.9] (16) .. controls({360/7}:.79) .. (26);
					\draw [ thick, blue, line width=0.15mm, opacity=0.9] (17) .. controls({2*360/7}:.79) .. (27);
					\draw [ thick, blue, line width=0.15mm, opacity=0.9] (18) .. controls({3*360/7}:.79) .. (28);
					\draw [ thick, blue, line width=0.15mm, opacity=0.9] (19) .. controls({4*360/7}:.79) .. (22);
					\draw [ thick, blue, line width=0.15mm, opacity=0.9] (20) .. controls({5*360/7}:.79) .. (23);
					\draw [ thick, blue, line width=0.15mm, opacity=0.9] (21) .. controls({6*360/7}:.79) .. (24);

				\end{tikzpicture}
			}
			\caption{$\widehat{BQ}(4,7)$.}
			\label{fig:BM4,7}
		\end{minipage}

	\end{figure}

	The fact that the underlying graph of $\widehat{BQ}(\ell,2k+1)$ is bipartite is easily observed. The parity of the levels gives a natural bipartition of the graph. We show that based on the choice of $k$ and $l$ this signed bipartite graph does not have a short negative cycle.
	
	\begin{proposition}\label{prop:Prop2}
		Given integers $l$ and $k$ where $l,k\geq 2$, the  length of the shortest negative cycle of $\widehat{BQ}(\ell,2k+1)$ is of length $\min\{2l, 2k+2\}$. 
	\end{proposition}

	\begin{proof}
		We first present two natural choices for a negative cycle, one of length $2k$ and another of length $2l$. The first is a negative cycle on the first two layers. Take a positive edge and connect its two ends with one of the two paths using only the negative edges that connect the two layers. This would result in  a negative cycle of length $2k+2$. The second negative cycle we consider is by taking a positive edge and connecting each of its ends to the vertex $u$ by a shortest path (all edges negative). One of these paths will be of length $l$ and the other would be of length $l-1$. Together with the first chosen edge itself then, they form a negative cycle of length $2l$. 
		
		It remains to show that the shortest of these two types of cycles gives us the negative girth. To that end, we will first show that a shortest negative cycle can only use one positive edge of $\widehat{BQ}(\ell,2k+1)$. Towards a contradiction, let $C$ be a negative cycle with more than two positive edges.  We aim to present a negative cycle $C'$ whose length is at most $|C|-2$. We take two positive edges of $C$ that come consecutively on the cyclic order. Assume $xy$ and $x'y'$ are these two edges and that $x'$ is followed by $y$ in the cyclic order of $C$ (that is to say, there is no positive edge in the $x'-y$ path in $C$). We remove the two positive edges $xy$ and $x'y'$ and the $x'y$ path connecting them in $C$, then we add a $xy'$ path in the first two layers which is not part of the cycle $C$ (which also has no positive edge). The result is a closed walk whose sign is the same as that of $C$, and whose length is $|C|-2$. But then this closed walk must contain a negative cycle, whose length then is also at most $|C|-2$, a contradiction.
		
		Finally, if $C$ is a cycle that uses exactly one negative edge, say $xy$, then the $x-y$ path $P_{xy}=C-{xy}$ either passes through $u$ in which case we have at least $2l$ edges in $C$, or the natural image of $P_{xy}$ to the cycle in between the first and second layers also connects $x$ to $y$. But the shortest such path is of length $2k+1$, thus $P_{xy}$ is of length at least $2k+1$, and the negative cycle is of length at least $2k+2$. 
	\end{proof}

	\subsection{$\widehat{BQ}(\ell,2k)$}	
	
	The third family of (signed) graphs we consider in this work are built quite similar to the previous construction. More precisely, given integers $\ell$ and $k$ satisfying $\ell, k \geq 2$, we define the (signed) graph  $\widehat{BQ}(\ell,2k)$ from $C_{\ell \times (2k)}$ as follows.   
	
	\begin{itemize}
		\item Edges of $C_{\ell \times 2k}$ are all negative.
		\item Connect $v_{_{1,j}}$ to $v_{_{1,j+k}}$  and $v_{_{2,j}}$ to $v_{_{2,j+k}}$  by  negative edges (Figure~\ref{fig:CylinderPlusMobuis2k}, right).
		\item Add a new vertex $u$ and connect it to each of the vertices $v_{_{\ell,j}}$, $j=1,\ldots, 2k$, with a negative edge (Figure~\ref{fig:CylinderPlusMobuis2k}, left).
	\end{itemize}
	
	\begin{figure}[h]
		\centering
		\begin{minipage}{.4\textwidth}

			\scalebox{1.3}{
				\tikzset{math3d/.style= {x={(2cm,0cm)}, y={(0cm,2cm)}}}
				\begin{tikzpicture}[math3d]
					\newcommand{\n}{16}
					\newcommand{\h}{7}
					\newcommand{\ch}{11}
					\newcommand{\rl}{1}
					\newcommand{\rh}{1}
					
					\path[draw, gray, line width=0.05mm] plot[domain=0:2*pi,samples=4*\n] ({\rh*cos(\x r)}, {\rh*sin(\x r)}, 3.5);
					
					\foreach \i in{1,2}{
						\path[draw, gray] plot[domain=0:2*pi,samples=4*\n] ({\rh*cos(\x r)}, {\rh*sin(\x r)}, .5*\i);
					}
					
					\foreach \i in {3,...,6}{
						\path[draw, dashed, gray] plot[domain=0:2*pi,samples=4*\n] ({\rh*cos(\x r)}, {\rh*sin(\x r)}, .5*\i);
					}
					
					\foreach \i in {1,3,5}{
						\foreach \t in {1,...,\n} {
							
							\draw [thick, red, line width=0.15mm, opacity=0.2] ( ({(\rl)*cos((2*\t+1)*pi/\n r)},{\rl*sin((2*\t+1)*pi/\n r)},.5*\i) -- ({\rh*cos((2*\t+1)*pi/\n r-\ch)},{\rh*sin((2*\t+1)*pi/\n r-\ch)},.5*\i+.5)  -- cycle;
							
							\draw [ thick, red, line width=0.15mm, opacity=0.2] ( ({(\rl)*cos((2*\t+1)*pi/\n r)},{\rl*sin((2*\t+1)*pi/\n r)},.5*\i) -- ({\rh*cos((2*\t+1)*pi/\n r+\ch)},{\rh*sin((2*\t+1)*pi/\n r+\ch)},.5*\i+.5)  -- cycle;
						}
					}
					
					\foreach \i in {2,4,6}{
						\foreach \t in {1,...,\n} {
							
							\draw [ thick, red, line width=0.15mm, opacity=0.2] ( ({(\rl)*cos((2*\t+1)*pi/\n r+\ch)},{\rl*sin((2*\t+1)*pi/\n r+\ch)},.5*\i) -- ({\rh*cos((2*\t+1)*pi/\n r)},{\rh*sin((2*\t+1)*pi/\n r)},.5*\i+.5)  -- cycle;
							
							\draw [ thick, red, line width=0.15mm, opacity=0.2] ( ({(\rl)*cos((2*\t+1)*pi/\n r-\ch)},{\rl*sin((2*\t+1)*pi/\n r-\ch)},.5*\i) -- ({\rh*cos((2*\t+1)*pi/\n r)},{\rh*sin((2*\t+1)*pi/\n r)},.5*\i+.5)  -- cycle;
						}
					} 
					
					\foreach \t in {1,...,\n} {
						\draw [ thick, red, line width=0.15mm, opacity=0.9] ( ({(cos((2*\t+14)*pi/\n r+\ch)},{sin((2*\t+14)*pi/\n r+\ch)},3.5) -- (-.5,-.5)  -- cycle;
					}
					\draw node[line width=0.05mm, circle, fill=white, minimum width=.5mm, inner sep=0mm, draw=black!80, minimum size=1mm] at (-.5,-.5) {};
					
				\end{tikzpicture}
			}
		\end{minipage}
		\begin{minipage}{.4\textwidth}
			\scalebox{1.3}{
				\tikzset{math3d/.style= {x={(2cm,0cm)}, y={(0cm,2cm)}}}
				\begin{tikzpicture}[math3d]
					\newcommand{\n}{16}
					\newcommand{\m}{8}
					\newcommand{\h}{7}
					\newcommand{\ch}{11}
					\newcommand{\rl}{1}
					\newcommand{\rh}{1}
					
					\foreach \i in{1,2}{
						\path[draw, gray] plot[domain=0:2*pi,samples=4*\n] ({\rh*cos(\x r)}, {\rh*sin(\x r)}, .5*\i);
					}
					
					\foreach \i in {3,...,\h}{
						\path[draw, dashed, gray] plot[domain=0:2*pi,samples=4*\n] ({\rh*cos(\x r)}, {\rh*sin(\x r)}, .5*\i);
						
					}
					\foreach \i in {1,3,5}{
						\foreach \t in {1,...,\n} {
							
							\draw [thick, red, opacity=0.2] ( ({(\rl)*cos((2*\t+1)*pi/\n r)},{\rl*sin((2*\t+1)*pi/\n r)},.5*\i) -- ({\rh*cos((2*\t+1)*pi/\n r-\ch)},{\rh*sin((2*\t+1)*pi/\n r-\ch)},.5*\i+.5)  -- cycle;
							
							\draw [ thick, red, opacity=0.2] ( ({(\rl)*cos((2*\t+1)*pi/\n r)},{\rl*sin((2*\t+1)*pi/\n r)},.5*\i) -- ({\rh*cos((2*\t+1)*pi/\n r+\ch)},{\rh*sin((2*\t+1)*pi/\n r+\ch)},.5*\i+.5)  -- cycle;
						}
					}
					
					\foreach \i in {2,4,6}{
						\foreach \t in {1,...,\n} {
							
							\draw [ thick, red, opacity=0.2] ( ({(\rl)*cos((2*\t+1)*pi/\n r+\ch)},{\rl*sin((2*\t+1)*pi/\n r+\ch)},.5*\i) -- ({\rh*cos((2*\t+1)*pi/\n r)},{\rh*sin((2*\t+1)*pi/\n r)},.5*\i+.5)  -- cycle;
							
							\draw [ thick, red, opacity=0.2] ( ({(\rl)*cos((2*\t+1)*pi/\n r-\ch)},{\rl*sin((2*\t+1)*pi/\n r-\ch)},.5*\i) -- ({\rh*cos((2*\t+1)*pi/\n r)},{\rh*sin((2*\t+1)*pi/\n r)},.5*\i+.5)  -- cycle;
						}
					}

					\foreach \t in {1,...,\m} {
						
						\draw [line width=0.15mm,  red, opacity=0.9] ( ({cos((2*\t+15)*pi/\n r+\ch)},{sin((2*\t+15)*pi/\n r+\ch)},1) .. controls
						({cos((2*\t+15)*pi/\n r+\ch)+cos((2*\t)*pi/\n r)},
						{sin((2*\t+15)*pi/\n r+\ch)+sin((2*\t)*pi/\n r)},0.5) .. ({cos((2*\t)*pi/\n r)},{sin((2*\t)*pi/\n r)},1);
						
					}
					
					\foreach \t in {1,2,...,\m} {
						
						\draw [line width=0.15mm,  red, opacity=0.9] ({cos((2*\t+16)*pi/\n r+\ch)},{sin((2*\t+16)*pi/\n r+\ch)},0.5)      .. controls
						({cos((2*\t+15)*pi/\n r+\ch)+cos((2*\t)*pi/\n r)},
						{sin((2*\t+15)*pi/\n r+\ch)+sin((2*\t)*pi/\n r)},0) ..
						({cos((2*\t+1)*pi/\n r)},{sin((2*\t+1)*pi/\n r)},0.5);
						
					}

				\end{tikzpicture}
			}
			
		\end{minipage}
		\caption{Construction of $\widehat{BQ}(\ell,2k)$.}
		\label{fig:CylinderPlusMobuis2k}
		
	\end{figure}
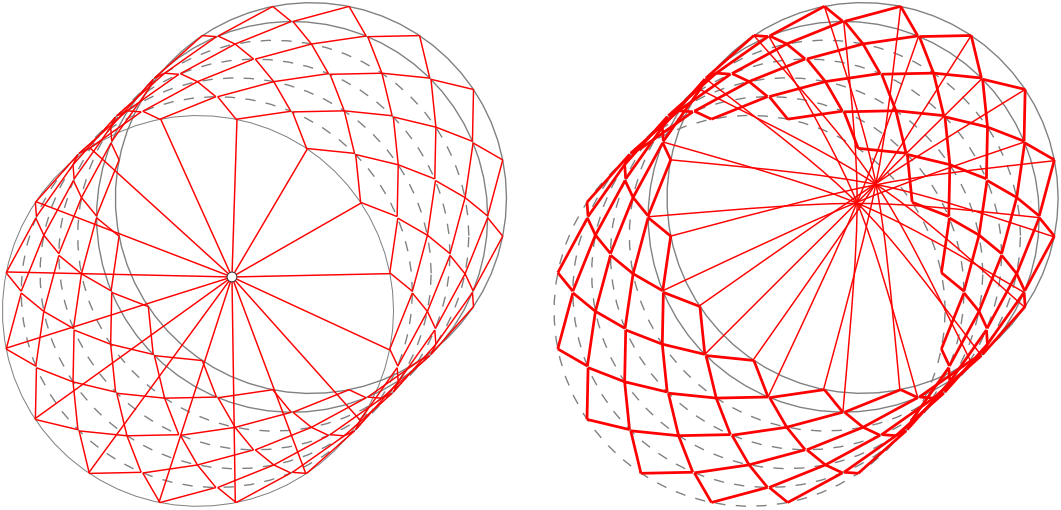
	
	As an example, the (signed) graphs $\widehat{BQ}(3,4)$ and $\widehat{BQ}(4,6)$ are depicted in Figures~\ref{fig:BM3,4} and \ref{fig:BM4,6} respectively.

	\begin{figure}[h]
		\begin{minipage}[t]{.4\textwidth}
			\centering
			\scalebox{1}{
				\begin{tikzpicture}
					
					\foreach \i in {1,2,3,4}
					{ \draw node[line width=0.05mm, circle, fill=white, minimum width=.5mm, inner sep=0mm, draw=black!80, minimum size=1mm]  (\i) at ({\i*90}:1.2){};}
					
					\foreach \i in {5,6,7,8}
					{ \draw node[line width=0.05mm, circle, fill=white, minimum width=.5mm, inner sep=0mm, draw=black!80, minimum size=1mm]  (\i) at ({\i*90+45}:2){};
					}
					
					\foreach \i in {9,10,11,12}
					{ \draw node[line width=0.05mm, circle, fill=white, minimum width=.5mm, inner sep=0mm, draw=black!80, minimum size=1mm]  (\i) at ({\i*90}:2.8){};}
					
					\draw node[line width=0.05mm, circle, fill=white, minimum width=.5mm, inner sep=0mm, draw=black!80, minimum size=1mm]  (0) at ({90}:0){};

					\draw [ thick, red, line width=0.15mm, opacity=0.9] (1) -- (0);
					\draw [ thick, red, line width=0.15mm, opacity=0.9] (2) -- (0);
					\draw [ thick, red, line width=0.15mm, opacity=0.9] (3) -- (0);
					\draw [ thick, red, line width=0.15mm, opacity=0.9] (4) -- (0);
					
					\draw [ thick, red, line width=0.15mm, opacity=0.9] (1) -- (5);
					\draw [ thick, red, line width=0.15mm, opacity=0.9] (2) -- (6);
					\draw [ thick, red, line width=0.15mm, opacity=0.9] (3) -- (7);
					\draw [ thick, red, line width=0.15mm, opacity=0.9] (4) -- (8);
					
					\draw [ thick, red, line width=0.15mm, opacity=0.9] (1) -- (8);
					\draw [ thick, red, line width=0.15mm, opacity=0.9] (2) -- (5);
					\draw [ thick, red, line width=0.15mm, opacity=0.9] (3) -- (6);
					\draw [ thick, red, line width=0.15mm, opacity=0.9] (4) -- (7);

					\draw [ thick, red, line width=0.15mm, opacity=0.9] (5) -- (9);
					\draw [ thick, red, line width=0.15mm, opacity=0.9] (6) -- (10);
					\draw [ thick, red, line width=0.15mm, opacity=0.9] (7) -- (11);
					\draw [ thick, red, line width=0.15mm, opacity=0.9] (8) -- (12);

					\draw [ thick, red, line width=0.15mm, opacity=0.9] (5) -- (10);
					\draw [ thick, red, line width=0.15mm, opacity=0.9] (6) -- (11);
					\draw [ thick, red, line width=0.15mm, opacity=0.9] (7) -- (12);
					\draw [ thick, red, line width=0.15mm, opacity=0.9] (8) -- (9);

					\draw  [ dashed, gray, line width=0.15mm, opacity=0.2] (0,0) circle (1.2cm);
					\draw  [ dashed, gray, line width=0.15mm, opacity=0.2] (0,0) circle (2cm);
					\draw  [ dashed, gray, line width=0.15mm, opacity=0.2] (0,0) circle (2.8cm);

					
					\draw [ thick, red, line width=0.15mm, opacity=0.9] (5) .. controls({0}:.49) .. (7);
					\draw [ thick, red, line width=0.15mm, opacity=0.9] (6) .. controls({72}:.49) .. (8);
					\draw [ thick, red, line width=0.15mm, opacity=0.9] (9) .. controls({144}:.49) .. (11);
					\draw [ thick, red, line width=0.15mm, opacity=0.9] (10) .. controls({216}:.49) .. (12);
					
				\end{tikzpicture}
			}
			\caption{$\widehat{BQ}(3,4)$.}
			\label{fig:BM3,4}
		\end{minipage}
		\centering
		\begin{minipage}[t]{.4\textwidth}
			\centering
			\scalebox{1}{
				\begin{tikzpicture}

					\foreach \i in {1,2,3,4,5,6}
					{ \draw node[line width=0.05mm, circle, fill=white, minimum width=.5mm, inner sep=0mm, draw=black!80, minimum size=1mm]  (\i) at ({\i*360/6}:1.2){};}
					
					\foreach \i in {7,8,9,10,11,12}
					{ \draw node[line width=0.05mm, circle, fill=white, minimum width=.5mm, inner sep=0mm, draw=black!80, minimum size=1mm]  (\i) at ({\i*360/6+180/6}:2){};
					}
					
					\foreach \i in {13,14,15,16,17,18}
					{ \draw node[line width=0.05mm, circle, fill=white, minimum width=.5mm, inner sep=0mm, draw=black!80, minimum size=1mm]  (\i) at ({\i*360/6}:2.8){};}
					
					\foreach \i in {19,20,21,22,23,24}
					{ \draw node[line width=0.05mm, circle, fill=white, minimum width=.5mm, inner sep=0mm, draw=black!80, minimum size=1mm]  (\i) at ({\i*360/6+180/6}:3.6){};}

					\draw [ thick, red, line width=0.15mm, opacity=0.9] (1) -- (7);
					\draw [ thick, red, line width=0.15mm, opacity=0.9] (2) -- (8);
					\draw [ thick, red, line width=0.15mm, opacity=0.9] (3) -- (9);
					\draw [ thick, red, line width=0.15mm, opacity=0.9] (4) -- (10);
					\draw [ thick, red, line width=0.15mm, opacity=0.9] (5) -- (11);
					\draw [ thick, red, line width=0.15mm, opacity=0.9] (6) -- (12);

					\draw [ thick, red, line width=0.15mm, opacity=0.9] (1) -- (12);
					\draw [ thick, red, line width=0.15mm, opacity=0.9] (2) -- (7);
					\draw [ thick, red, line width=0.15mm, opacity=0.9] (3) -- (8);
					\draw [ thick, red, line width=0.15mm, opacity=0.9] (4) -- (9);
					\draw [ thick, red, line width=0.15mm, opacity=0.9] (5) -- (10);
					\draw [ thick, red, line width=0.15mm, opacity=0.9] (6) -- (11);

					\draw [ thick, red, line width=0.15mm, opacity=0.9] (7) -- (13);
					\draw [ thick, red, line width=0.15mm, opacity=0.9] (8) -- (14);
					\draw [ thick, red, line width=0.15mm, opacity=0.9] (9) -- (15);
					\draw [ thick, red, line width=0.15mm, opacity=0.9] (10) -- (16);
					\draw [ thick, red, line width=0.15mm, opacity=0.9] (11) -- (17);
					\draw [ thick, red, line width=0.15mm, opacity=0.9] (12) -- (18);

					\draw [ thick, red, line width=0.15mm, opacity=0.9] (7) -- (14);
					\draw [ thick, red, line width=0.15mm, opacity=0.9] (8) -- (15);
					\draw [ thick, red, line width=0.15mm, opacity=0.9] (9) -- (16);
					\draw [ thick, red, line width=0.15mm, opacity=0.9] (10) -- (17);
					\draw [ thick, red, line width=0.15mm, opacity=0.9] (11) -- (18);
					\draw [ thick, red, line width=0.15mm, opacity=0.9] (12) -- (13);

					\draw [ thick, red, line width=0.15mm, opacity=0.9] (13) -- (19);
					\draw [ thick, red, line width=0.15mm, opacity=0.9] (13) -- (24);
					\draw [ thick, red, line width=0.15mm, opacity=0.9] (14) -- (20);
					\draw [ thick, red, line width=0.15mm, opacity=0.9] (14) -- (19);
					\draw [ thick, red, line width=0.15mm, opacity=0.9] (15) -- (21);
					\draw [ thick, red, line width=0.15mm, opacity=0.9] (15) -- (20);
					\draw [ thick, red, line width=0.15mm, opacity=0.9] (16) -- (22);
					\draw [ thick, red, line width=0.15mm, opacity=0.9] (16) -- (21);
					\draw [ thick, red, line width=0.15mm, opacity=0.9] (17) -- (23);
					\draw [ thick, red, line width=0.15mm, opacity=0.9] (17) -- (22);
					\draw [ thick, red, line width=0.15mm, opacity=0.9] (18) -- (24);
					\draw [ thick, red, line width=0.15mm, opacity=0.9] (18) -- (23);
					\draw  [ dashed, gray, line width=0.15mm, opacity=0.2] (0,0) circle (1.2cm);
					\draw  [ dashed, gray, line width=0.15mm, opacity=0.2] (0,0) circle (2cm);
					\draw  [ dashed, gray, line width=0.15mm, opacity=0.2] (0,0) circle (2.8cm);
					\draw  [ dashed, gray, line width=0.15mm, opacity=0.2] (0,0) circle (3.6cm);
					
					\draw node[line width=0.05mm, circle, fill=white, minimum width=.5mm, inner sep=0mm, draw=black!80, minimum size=1mm]  (29) at ({90}:0){};
					
					\draw [ thick, red, line width=0.15mm, opacity=0.9] (1) -- (29);
					\draw [ thick, red, line width=0.15mm, opacity=0.9] (2) -- (29);
					\draw [ thick, red, line width=0.15mm, opacity=0.9] (3) -- (29);
					\draw [ thick, red, line width=0.15mm, opacity=0.9] (4) -- (29);
					\draw [ thick, red, line width=0.15mm, opacity=0.9] (5) -- (29);
					\draw [ thick, red, line width=0.15mm, opacity=0.9] (6) -- (29);

					\draw [ thick, red, line width=0.15mm, opacity=0.9] (15) .. controls({4*360/12}:-.79) .. (18);
					\draw [ thick, red, line width=0.15mm, opacity=0.9] (16) .. controls({5*360/12}:.7) .. (13);
					\draw [ thick, red, line width=0.15mm, opacity=0.9] (17) .. controls({6*360/12}:-.79) .. (14);
					\draw [ thick, red, line width=0.15mm, opacity=0.9] (19) .. controls({1*360/6}:.79) .. (22);
					\draw [ thick, red, line width=0.15mm, opacity=0.9] (20) .. controls({2*360/6}:-.79) .. (23);
					\draw [ thick, red, line width=0.15mm, opacity=0.9] (21) .. controls({3*360/6}:.79) .. (24);

				\end{tikzpicture}
			}
			\caption{$\widehat{BQ}(4,6)$.}
			\label{fig:BM4,6}
		\end{minipage}
	\end{figure}

	\begin{proposition}
		Given integers $l$ and $k$, where $l,k \geq2$, the shortest negative cycle of $\widehat{BQ}(l, 2k)$ is of length $\min\{2l-1, 2k+1\}$.
	\end{proposition}
	\begin{proof}
		A cycle of $\widehat{BQ}(l, 2k)$ which does not contain any step of the M\"obius ladder induced by the first two layers is even. That is to say any odd cycle has at least one step of this  M\"obius ladder. A step of the M\"obius ladder together with one of the two paths that are connecting the end vertices of this step through the first two layers form an odd cycle of length $2k+1$. Also, there is another natural choice for an odd cycle constituted by this step and the shortest path which contains the universal vertex $u$ connecting the end vertices of this step. This cycle is of length $2l-1$. In a similar way as in the proof of Proposition~\ref{prop:Prop2} one may conclude that one of these two odd cycles of $\widehat{BQ}(l,2k)$ is the shortest.  
	\end{proof}
	
	The two constructions $\widehat{BQ}(\ell, 2k+1)$ and $\widehat{BQ}(\ell, 2k)$ can be defined uniformly as follows. Starting with an $i$-star ($i=2k+1$ or $i=2k$) on the projective plane, we complete it to quadrangulation of the planar part except for the vertices on the outer layer which are at distance $\ell$ or $\ell-1$ from the center of the star, and assign a negative sign to everything so that all facial 4-cycles are positive. We then complete the outer layer to a M\"obuis ladder, choosing signs for the crossing edges so that all faces are positive 4-cycles but the non-contractible cycles are negative.
	
	We view this class of signed graphs as \textit{Basic Qudrangulations} of the projective plane and thus use the notation $\widehat{BQ}(\ell, i)$.  
	
	\subsection{$\widehat{BM}(\ell,2k)$}
	
	The last construction we present here, $\widehat{BM}(\ell,2k)$, is built from $C_{\ell \times 2k}$ as follows. Taking all the edges of this graph as negative edges, on the last layer of the cylinder, as in the other cases, we add a (universal) vertex which is joined to all vertices of this layer with negative edges. On the first layer we add a set $\{u_{1}, \dots, u_{k}\}$ of vertices, then join each $u_i$, $i=1, \dots k$, to $v_{1, i}$ and $v_{1, i+1}$ with negative edges and to $v_{1, i+k}$ and $v_{1, i+k+1}$ with positive edges. See Figure~\ref{fig:SecondBipartite} for a depiction.
	
	We leave it to the reader to check the following.
	
	\begin{proposition}
		Given integers $l,k \geq 2$, the shortest negative cycle of the signed bipartite graph $\widehat{BM}(l,2k)$ is of length $min\{2l+2,2k\}$. 
	\end{proposition} 
	
	In particular, $\widehat{BM}(k-1,2k)$ has $2k^2-k+1$ vertices and its shortest negative cycle is of length $2k$.

	\begin{figure}[ht]
		\centering
		
		\begin{minipage}{.4\textwidth}

			\scalebox{1.3}{
				\tikzset{math3d/.style= {x={(2cm,0cm)}, y={(0cm,2cm)}}}
				\begin{tikzpicture}[math3d]
					\newcommand{\n}{16}
					\newcommand{\h}{7}
					\newcommand{\ch}{12}
					\newcommand{\rl}{1}
					\newcommand{\rh}{1}

					\path[draw, gray] plot[domain=0:2*pi,samples=4*\n] ({\rh*cos(\x r)}, {\rh*sin(\x r)}, 3.5);
					
					\foreach \i in {1,...,6}{
						\path[draw, dashed, gray] plot[domain=0:2*pi,samples=4*\n] ({\rh*cos(\x r)}, {\rh*sin(\x r)}, .5*\i);
					}

					\foreach \i in {1,3,5}{
						\foreach \t in {1,...,\n} {
							
							\draw [thick, red, opacity=0.2] ( ({(\rl)*cos((2*\t+1)*pi/\n r)},{\rl*sin((2*\t+1)*pi/\n r)},.5*\i) -- ({\rh*cos((2*\t+1)*pi/\n r-\ch)},{\rh*sin((2*\t+1)*pi/\n r-\ch)},.5*\i+.5)  -- cycle;
							
							\draw [ thick, red, opacity=0.2] ( ({(\rl)*cos((2*\t+1)*pi/\n r)},{\rl*sin((2*\t+1)*pi/\n r)},.5*\i) -- ({\rh*cos((2*\t+1)*pi/\n r+\ch)},{\rh*sin((2*\t+1)*pi/\n r+\ch)},.5*\i+.5)  -- cycle;
						}

					}
					
					\foreach \i in {2,4,6}{
						\foreach \t in {1,...,\n} {
							
							\draw [ thick, red, opacity=0.2] ( ({(\rl)*cos((2*\t+1)*pi/\n r+\ch)},{\rl*sin((2*\t+1)*pi/\n r+\ch)},.5*\i) -- ({\rh*cos((2*\t+1)*pi/\n r)},{\rh*sin((2*\t+1)*pi/\n r)},.5*\i+.5)  -- cycle;
							
							\draw [ thick, red, opacity=0.2] ( ({(\rl)*cos((2*\t+1)*pi/\n r-\ch)},{\rl*sin((2*\t+1)*pi/\n r-\ch)},.5*\i) -- ({\rh*cos((2*\t+1)*pi/\n r)},{\rh*sin((2*\t+1)*pi/\n r)},.5*\i+.5)  -- cycle;
						}
					}

					\foreach \t in {1,...,\n} {
						
						\draw [  red, opacity=1] ( ({(cos((2*\t+14)*pi/\n r+\ch)},{sin((2*\t+14)*pi/\n r+\ch)},3.5) -- (-.5,-.5)  -- cycle;
						
					}
					\draw node[circle, red, inner sep=0mm, draw=black!80, fill=white, minimum size=1mm] at (-.5,-.5) {};
				\end{tikzpicture}
			}
		\end{minipage}
		\begin{minipage}{.4\textwidth}
			\scalebox{1.3}{
				\tikzset{math3d/.style= {x={(2cm,0cm)}, y={(0cm,2cm)}}}
				\begin{tikzpicture}[math3d]
					\newcommand{\n}{16}
					\newcommand{\h}{7}
					\newcommand{\hl}{8}
					\newcommand{\ch}{12}
					\newcommand{\rl}{1}
					\newcommand{\rh}{1}
					
					
					\foreach \i in {1,...,\h}{
						\path[draw, dashed, gray] plot[domain=0:2*pi,samples=4*\n] ({\rh*cos(\x r)}, {\rh*sin(\x r)}, .5*\i);

					}
					\foreach \i in {1,3,5}{
						\foreach \t in {1,...,\n} {
							
							\draw [thick, red, opacity=0.2] ( ({(\rl)*cos((2*\t+1)*pi/\n r)},{\rl*sin((2*\t+1)*pi/\n r)},.5*\i) -- ({\rh*cos((2*\t+1)*pi/\n r-\ch)},{\rh*sin((2*\t+1)*pi/\n r-\ch)},.5*\i+.5)  -- cycle;
							
							\draw [ thick, red, opacity=0.2] ( ({(\rl)*cos((2*\t+1)*pi/\n r)},{\rl*sin((2*\t+1)*pi/\n r)},.5*\i) -- ({\rh*cos((2*\t+1)*pi/\n r+\ch)},{\rh*sin((2*\t+1)*pi/\n r+\ch)},.5*\i+.5)  -- cycle;
						}

					}
					
					\foreach \i in {2,4,6}{
						\foreach \t in {1,...,\n} {
							
							\draw [ thick, red, opacity=0.2] ( ({(\rl)*cos((2*\t+1)*pi/\n r+\ch)},{\rl*sin((2*\t+1)*pi/\n r+\ch)},.5*\i) -- ({\rh*cos((2*\t+1)*pi/\n r)},{\rh*sin((2*\t+1)*pi/\n r)},.5*\i+.5)  -- cycle;
							
							\draw [ thick, red, opacity=0.2] ( ({(\rl)*cos((2*\t+1)*pi/\n r-\ch)},{\rl*sin((2*\t+1)*pi/\n r-\ch)},.5*\i) -- ({\rh*cos((2*\t+1)*pi/\n r)},{\rh*sin((2*\t+1)*pi/\n r)},.5*\i+.5)  -- cycle;
						}
					}

						%
						%
					
					\foreach \t in {1,...,\n} {
						
						\draw node[circle, red, inner sep=0mm, draw=black!80, fill=white, minimum size=.5mm] at (({cos((2*\t+14)*pi/\n r+\ch)},{sin((2*\t+14)*pi/\n r+\ch)},.5)  (x\t) {};

					}
					
					\foreach \t in {1,...,\hl} {
						
						\draw node[circle, red, inner sep=0mm, draw=black!80, fill=white, minimum size=1mm] at (({cos((1.725*\t+12)*pi/\n r+\ch)/1.6},{sin((1.725*\t+12)*pi/\n r+\ch)/1.8})  (y\t){};

					}
					
					\foreach \i/\j in { 15/1, 16/1,16/2, 1/2,1/3,2/3,2/4, 3/4, 3/5, 4/5, 4/6, 5/6, 5/7, 6/7, 6/8, 7/8} {
						
						\draw[line width=0.2mm, red ] (x\i) -- (y\j) ;}
					
					\foreach \i/\j in { 15/8, 14/8, 14/7, 13/7, 13/6, 12/6, 12/5, 11/5, 11/4, 10/4, 10/3, 9/3, 9/2, 8/2, 8/1, 7/1} {
						
						\draw[line width=0.2mm,blue] (x\i) -- (y\j) ;}
					
				\end{tikzpicture}
			}
		\end{minipage}
		\caption{ $\widehat{BM}(\ell,2k)$ }
		\label{fig:SecondBipartite}
	\end{figure}
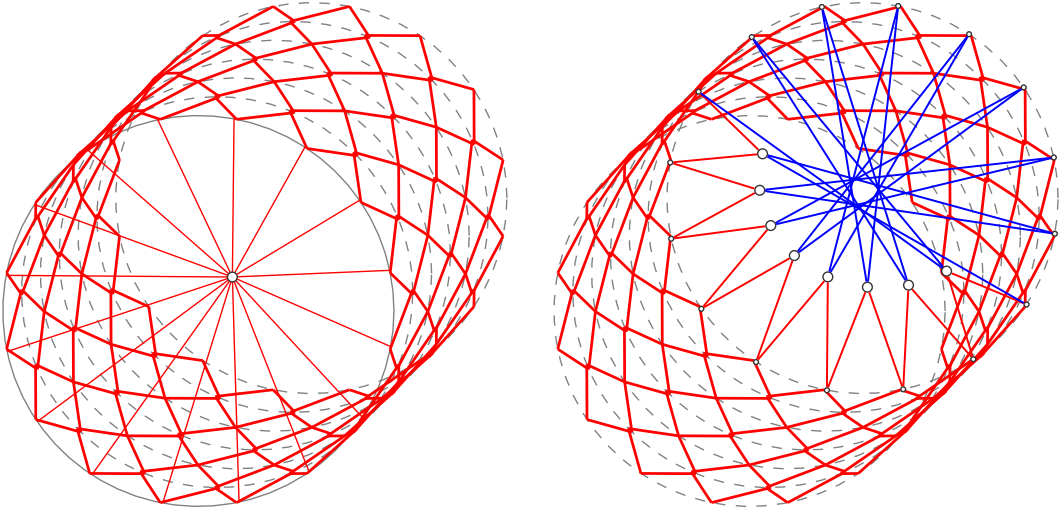

	\section{Winding number and coloring}\label{sec:Winding-Coloring}
	
	Given a simple closed curve $\gamma$ on the plane, and a continuous mapping $\varphi$ of $\gamma$ to $O_r$, we define the \textit{winding number} of the pair $(\gamma, \varphi)$ to be the winding number of the curve $\varphi(\gamma)$ with center of $O_r$ considered as the center of the plane. Intuitively speaking, $(\gamma, \varphi)$ tells us how many times the curve $\gamma$ is wrapped around $O_r$ in the clockwise direction noting that a negative number reflects an anticlockwise mapping. This value then will be denoted by $\omega(\gamma, \varphi)$.
	
	A mapping $c$ of the vertices of the cycle $C_n$ to the points of $O_r$ can be extended to a continuous mapping of $C_n$ to $O_r$ with the former being viewed as the closed curve or the polygon. There are $2^n$ natural ways to do this. For each pair $v_i,v_{i+1}$ of the vertices of $C_n$, the pair $c(v_i), c(v_{i+1})$ partitions the circle $O_r$ into two parts. The segment of the polygon that represents the edge $v_iv_{i+1}$ can be projected into one of these two parts. We note that $c$ is allowed to map several vertices of $C_n$ to the same point and that even if $v_i$ and $v_{i+1}$ are mapped to the same point, in our view, they partition the circle $O_r$ into two parts: a part of length 0 and a part of length $r$.

	These $2^n$ extensions are in a one-to-one correspondence with the $2^n$ possible orientations of $C_n$: orient the edge $v_iv_{i+1}$ in such a way that the mapping follows the clockwise direction of $O_r$.
	
	Given a coloring $c$ of the vertices of the cycle $C_n$, two extensions of $c$ to a mapping of the polygon to $O_r$ are of special importance. The first is the extension corresponding to the directed cycle $C_n$. Here $v_iv_{i+1}$ is mapped to the part of the circle where $c(v_{i+1})$ follows $c(v_i)$ in the clockwise direction. Let us denote this extension by $c^D$. A trivial observation here is that the winding number of $(C_n, c^D)$ is never 0.

	The other natural extension is to choose the shortest of the two parts of the circle determined by $c(v_i)$ and $c(v_{i+1})$ and project the line $v_{i}v_{i+1}$ onto it. The orientation corresponding to this extension then depends on whether $c(v_i)$ is the start or the end of this shorter part of the circle with respect to the clockwise orientation. We denote this extension by $c^{sh}$ and observe that this extension may result in winding number 0 for some choices of $c$ (and $r$). 
	
	Given the cycle $C_n$, a mapping $c$ of its vertices to $O_r$ and an extension $\varphi$ of $c$ to the polygon, a combinatorial way to compute $\omega(C_n, \varphi)$ is as follows: take an (open) interval $I$ on $O_r$ which does not contain any image of the vertices of $C_n$. Then in an extension $\varphi$ of $c$ to a mapping of the polygon to $O_r$, each edge of $C_n$ either traverses $I$ completely or does not touch any point of it. Now the winding number $\omega(C_n, \varphi)$ is the number of edges that traverse $I$ in the clockwise direction minus the number of edges that traverse it in the anticlockwise direction (and thus independent of the choice of I).
	
	Let $c$ be a mapping of the vertices of a cycle $C_n$ to the circle $O_r$. Consider the continuous mapping $(C_n, c^D)$ and an (open) interval $I$ of $O_r$ which does not contain any point $c(v_i)$. Color the edges of $C_n$ with two colors, say green and orange, as follows: if the image of an edge $e$ under $c^D$ contains $I$, then color it green, otherwise, color it orange. We are interested in the pairs of consecutive edges $v_{i-1}v_{i}$ and $v_{i}v_{i+1}$, which are colored differently. If in such a pair, the first edge is colored green, then in the next pair of this sort (next in the cyclic order of indices), the first edge must be orange and vice versa. Thus, the total number of such pairs is even, that is regardless of the choices of $n$ and $c$. 
	
	To use this observation, we will work with certain types of mappings $c$. We say a mapping $c$ of the vertices of $C_n$ to the points on $O_r$ is \textit{far-polar} if the followings hold: for each $i$ the pair of the points $c(v_{i-1})$ and $c(v_{i+1})$ on $O_r$ partitions $O_r$ into two unequal parts and that $c(v_i)$ is on the larger of the two parts. More generally, a mapping $\phi$ of the vertices of a graph $G$ to the circle $O_r$ is called far-polar if for each vertex $x$ of $G$ there is a diameter $D_{x}$ which separates $\phi(x)$ from $\phi(y)$ for all neighbors $y$ of $x$. 
	
	In the following, we present how the condition of $c$ being a far-polar mapping provides a connection between $c^D$ extension of $c$ on $C_n$ and $c^{sh}$ extension of the mapping $c$ on $C_n^{\#2}$.

	\begin{lemma}\label{lem:NonCrossingAlwaysEven}
		Let $c$ be a far-polar mapping of $C_n$ to $O_r$ and let $I$ be an interval of $O_r$ which does not contain any $c(v_i)$. Then in the extension $c^{sh}$ of a mapping of the one or two cycles in $C_n^{\#2}$ to $O_r$, the number of edges $v_{i-1}v_{i+1}$ that does not cross over $I$ is an even number.
	\end{lemma}
	
	\begin{proof}
		Consider three consecutive vertices $v_{i-1}, v_i, v_{i+1}$ of the cycle. If, following the $c^D$ extension of $C_n$, both edges $v_{i-1}v_{i}$ and $v_{i}v_{i+1}$ are colored orange, that is to say, in the extension they do not pass through $I$, then since $c$ is far-polar, $c(v_i)$ must be on the longer of $c(v_{i-1})c(v_{i+1})$ or  $c(v_{i+1})c(v_{i-1})$ and thus $I$ is on the shorter part. Thus in the extension $c^{sh}$ of $C_n^{\#2}$, $c(v_{i-1})c(v_{i+1})$ passes through $I$. Similarly, if both edges $v_{i-1}v_{i}$ and $v_{i}v_{i+1}$ are colored green, then $c(v_{i-1})c(v_{i+1})$ passes through $I$. On the other hand, if one of the edges is green and the other orange, then together, they must cover more than half of the $O_r$. Implying that $c(v_{i-1})c(v_{i+1})$ does not pass through $I$.

		Overall the number of edges of $C_n^{\#2}$ that do not pass through $I$ in $c^{sh}$ extension is the number of vertices of $C_n$ incident with both green and orange edges, where the colors are determined by the extension $c^D$ of $C_n$. The number of such pairs then must be even as it is observed above.  
	\end{proof}
	
	Next we give a useful observation on far-polar mapping. 

 \begin{observation}\label{obs : far-polar coloring}
     If $r<4$ and $c$ is a circular $r$-coloring of $C_n$, then it is, in particular, a far-polar mapping of $C_n$. That is because for three consecutive vertices $v_{i-1}$, $v_{i}$, and $v_{i+1}$, having partitioned $O_r$ to two parts based on $c(v_{i-1})$ and $c(v_{i+1})$, the part that contains $c(v_i)$ must be of length at least 2. As $r<4$, this must be the larger part.
 \end{observation}
  
  Now, we have the following two consequences depending on the parity of $n$.

	\begin{lemma}\label{lem:WC_2k^2}
		Let $c$ be a circular $r$-coloring of an even cycle $(C_n,-)$. Let $c_o$ (resp. $c_e$) be its restriction on the vertices with odd (resp. even) indices.  Then the winding numbers of $(C^{\#2o}, c^{sh}_o)$ and $(C^{\#2e}, c^{sh}_e)$ are of the same parity.
	\end{lemma} 
	
	\begin{proof}
		That is because after choosing a suitable interval $I$, by Lemma~\ref{lem:NonCrossingAlwaysEven}, the total number of edges of $C^{\#2}$ that does not cross over $I$ in the extension $c^{sh}$ is even. As the total number of edges is also even (that is $n$), the number of edges of $C^{\#2}$ that cross over $I$ is also even. However, the winding number of each of $(C^{\#2o}, c^{sh}_o)$ and $(C^{\#2e}, c^{sh}_e)$, which is the difference of the number of edges crossing $I$ in the clockwise direction and the number of edges crossing it in the anticlockwise direction, has the same parity as the total number of the edges of the cycle in consideration that cross over $I$ (in the $c^{sh}$ extension). This proves our claim as the sum of the two winding numbers is an even number.
	\end{proof}
	
	Using this lemma, we can build a cylinder of many layers, as shown in the example of Figure~\ref{fig:ell-Times-2k+1}, with the property that in any circular $r$-coloring $c$ of the red graph ($r<4$), all of the dashed grey cycles must have winding numbers of the same parity. Observe that in this construction, the zigzag red cycle between two consecutive layers is an even cycle, and its exact square consists of the two grey cycles presenting the two layers. If we then add structures to the two ends in such a way that one force an odd winding number on one of the grey cycles and the other forces an even winding number on another one of them, then the result would be a graph which admits no circular $r$-coloring for $r<4$.
	
	A basic method to achieve these conditions 
	is presented next.

	\begin{lemma}\label{coro:far-polarC_2k+1^2}
		Given an odd integer $n$, a positive real number $r$, and a far-polar mapping $c$ of $C_n$ to $O_r$, the winding number $\omega(C^{\#2}_n, c^{sh})$ is an odd number.
	\end{lemma}
	
	\begin{proof} 
		By Lemma~\ref{lem:NonCrossingAlwaysEven}, the total number of edges of $C^{\#2}_n$ that does not cross over $I$ is even. As $n$ is an odd number, $C^{\#2}_n$ is isomorphic to $C_n$, and, hence, the number of edges crossing over $I$ is odd. This is the sum of the number of edges crossing over $I$ in the clockwise direction and in the anticlockwise direction. Thus the winding number, which is the difference between these two numbers, is also an odd number. 
	\end{proof}
	
	Applying this lemma on circular $r$-coloring for $r<4$ we have the following.
	
	\begin{lemma}\label{coro:WC_2k+1^2}
		Given an odd integer $n$, $n=2k+1$, a real number $r$ satisfying $2+\frac{1}{k}\leq r <4$, and a circular $r$-coloring  $c$ of $(C_n,-)$, the winding number $\omega(C^{\#2}_n, c^{sh})$ is an odd number.
	\end{lemma}

	\begin{proof} 

            The proof follows from the Observation \ref{obs : far-polar coloring} and by Lemma \ref{coro:far-polarC_2k+1^2}.
	\end{proof}
	
	\begin{observation}\label{obs:AddingUniversalVertex}
		Let $G$ be the star $K_{1,n}$ with $u$ being the central vertex and $A$ being the independent set of order $n$. Let $c$ be a circular $r$-coloring of $G$ with $r<4$. Then for any cycle $C$ with $V(C)=A$ the winding number of $(C, c^{sh})$ is 0.  
	\end{observation}
	
	This is observed by taking a small interval $I$ sufficiently close to $c(u)$ and noting that first of all, a vertex of $C$ cannot be mapped to $c(u)$; secondly,  since $r<4$, for any pair $x$ and $y$ of vertices in $A$ in the partition of $O_r$ to two parts by $c(x)$ and $c(y)$, the part containing $c(u)$ is of length at least 2 and thus it is the larger of the two, meaning in the shortest extension, $c(x)c(y)$ will never cross over $I$.
	
	We may now give a new proof of the following theorem.
	
	\begin{theorem}\label{thm:ChiMlC2k+1}
		For any positive integers $\ell$ and $k$, we have $\chi_c(M_{\ell}(C_{2k+1}))=4$.
	\end{theorem}
	
	\begin{proof}
		It is enough to observe that $M_{\ell}(C_{2k+1})$ is obtained from the $l\times (2k+1)$ cylindrical grid of Figure~\ref{fig:ell-Times-2k+1} by adding 
  nearly diagonal edges to the bottom layer (that is connecting pairs at a distance $k$ of the grey cycle) and adding a universal vertex to the top layer (as mentioned in the previous section). As any circular $r$-coloring with $r<4$ is also far-polar, any such a coloring would imply an odd winding number for the layers in $c^{sh}$ extension from one end and an even winding number for the layers from the other end. So a proper mapping to $O_r$ where $r<4$ is impossible. On the other hand, one can easily color $M_{\ell}(C_{2k+1})$ with $4$ colors, which gives the upper bound $4$ on the circular chromatic number as well.
	\end{proof}

	Next, we show that $\widehat{BQ}(\ell, 2k+1)$ shares the same property. We will note later that Theorem~\ref{thm:ChiMlC2k+1} follows from the next theorem.
	
	\begin{theorem}\label{thm:ChiBMl2k+1}
		For given positive integers $\ell$ and $k$, satisfying $l \geq 2$ and $k \geq 1$, we have $\chi_c(\widehat{BQ}(\ell, 2k+1))=4$.
	\end{theorem}
	
	\begin{proof}
		Towards a contradiction, let $c$ be a circular $r$-coloring of $\widehat{BQ}(\ell, 2k+1)$ with $r<4$. We will have a contradiction if we show that the cycle $C'$ formed on $v_{_{1,1}}v_{_{1,2}}\cdots v_{_{1,2k+1}}$ in this cyclic order has an odd winding number under the mapping $c^{sh}$ (restricted on the vertices of this cycle). We emphasize that edges of $C'$ are not in $\widehat{BQ}(\ell, 2k+1)$.
		
		To this end we first consider another cycle, $C^{\star}$, (also not part of our graph) by considering the following sequence of vertices of the first layer of $\widehat{BQ}(\ell, 2k+1)$: $v_{_{1,1}}v_{_{1,k+1}}v_{_{1,2}}v_{_{1,k+2}}\cdots v_{_{1,k+2}}$. Note that in this cycle $v_{_{1,j}}$ is followed by $v_{_{1,j+k}}$ where the addition is taken $\pmod{2k+1}$. We may also note that this is the diagonally drawn cycle on the first layer of Figure~\ref{fig:CylinderPlusCycle} (right). 
		
		Our claim is that the mapping $c$, viewed as a mapping of the vertices of $C^{\star}$ to $O_r$, is a far-polar mapping. 
		Toward proving the claim, we consider $c(v_{_{1,j}})$, $c(v_{_{1,j+k}})$, and $c(v_{_{1,j+2k}})$. The first observation is that since $v_{_{2,j+2k}}$ is adjacent to both $v_{_{1,j}}$ and $v_{_{1,j+2k}}$ with negative edges, the points $c(v_{_{1,j}})$ and $c(v_{_{1,j+2k}})$ of $O_r$ partition $O_r$ in such a way that the part containing $c(v_{_{2,j+2k}})$ is at least 2. As $r<4$, it follows that $c(v_{_{2,j+2k}})$ is on the larger part of $O_r$ when it is partitioned by $c(v_{_{1,j}})$ and $c(v_{_{1,j+2k}})$. It remains to show that $1) toc(v_{_{1,j+k}})$ is also on the same part. If not, that is if $c(v_{_{1,j+k}})$ is on the shorter side of $O_r$, then one of the arcs $c(v_{_{1,j+k}})c(v_{_{2,j+2k}})$ and  $c(v_{_{2,j+2k}})c(v_{_{1,j+k}})$ contains the shorter side of $c(v_{_{1,j}})c(v_{_{2,j+2k}})$ and the other contains the shorter side of $c(v_{_{1,j+1}})c(v_{_{2,j}})$. As each of these shorter arcs is of length at least one, we conclude that the distance of $c(v_{_{1,j+k}})$ and  $c(v_{_{2,j+2k}})$ is at least one. However, since $c$ is a circular $r$-coloring where $r<4$ and $v_{_{1,j+k}}v_{_{2,j+2k}}$ is a positive edge, they should be at a distance at most $\frac{r}{2}-1<1$, a contradiction.
		
		Finally, observing that $C'$ is the exact square of $C^{\star}$, and by Lemma~\ref{coro:far-polarC_2k+1^2}, we conclude that the winding number of $C'$ is odd. 
	\end{proof}

	 To prove that $\chi_{c}(\widehat{BQ}(\ell,2k))=4$, we need a few more lemmas.
 

	\begin{lemma}\label{lem: C_4}
		Let $c$ be a far-polar mapping of $C_4$ to $O_r$. Then the winding number $w(C_4, c^D)$ is $2$.
	\end{lemma}
	
	\begin{proof}
		The points $c(v_1)$ and $c(v_3)$ partition $O_r$ into two unequal parts. Since $c$ is a far-polar mapping, we know that $c(v_2)$ and $c(v_4)$ both should be on the larger of these two parts. Without loss of generality, we may assume that images of the vertices are in the following cyclic order: $c(v_1), c(v_3), c(v_2), c(v_4)$. Let $I$ be an interval of $O_r$ in $c(v_3)c(v_2)$ that does not contain any $c(v_i)$ (the winding number is independent of the choice of $I$). Then following the orientation of $C_4$ and clockwise direction of $O_r$, the arcs $c(v_1)c(v_2)$ and $c(v_3)c(v_4)$ contain the interval $I$ while the arcs $c(v_2)c(v_3)$ and $c(v_4)c(v_1)$ do not intersect it. 
	\end{proof}
	
	\begin{lemma}\label{lem:o-g parity}
		Let $c$ be a far-polar mapping of $C_{2k}$ to $O_r$ and let the edges of $C_{2k}$ be $e_1, e_2, \dots, e_{2k}$. The number of edges colored green in the extension $c^D$ of $c$ has the same parity as the number of odd (or even) indexed vertices being incident to both green and orange edges.
	\end{lemma}
	
	\begin{proof}
		Let $J$ be the set of vertices of $C_{2k}$ incident to both green and orange edges. Let $J_o$ and $J_e$ be the partition of $J$ into odd and even indexed vertices (a natural bipartition of on $C_{2k}$).
		
		Consider a maximal green path in $C_{2k}$. Thus the two ends of each such path are in $J$. Moreover, if the length of the path is even, then both ends of the path belong to the same subset $J_o$ or $J_e$ of $J$. Thus each even green path contributes $0$ to one of $J_o$ or $J_e$ and $2$ to the other. If the length of the path is odd, then one of its ends is in $J_o$ and the other is in $J_e$, thus contributing $1$ to each of these two sets. The claim then follows as the odd length green-paths determine the parity of the total number of green edges.
	\end{proof}
	
	\begin{lemma}\label{lem: 4k zigzag greens and winding number}
		Let $c$ be a far-polar mapping of the cycle $C_{4k}$. The number of edges colored green in the extension $c^D$ of $c$ and the winding number $w(C_{4k}^{\#2e}, c^{sh})$ (or similarly $w(C_{4k}^{\#2o}, c^{sh})$) are of the same parity.
		
	\end{lemma}
	
	\begin{proof}

		For each edge $v_{i-1}v_{i+1}$ of $C_{4k}^{\#2e}$ there is an odd indexed vertex $v_i$ of $C_{4k}$ corresponding to it. (And similarly, for each edge $v_{i-1}v_{i+1}$ of $C_{4k}^{\#2o}$ there is an even indexed vertex $v_i$ of $C_{4k}$.) As we have observed before, an edge $v_{i-1}v_{i+1}$ in $C^{\#2}$ does not cross over the interval $I$ in the $c^{sh}$ extension if and only if the edges incident to $v_i$ are colored differently (i.e. one of $v_{i-1}v_i$ and $v_iv_{i+1}$ is green, the other is orange). So the number of non-crossing edges in $C_{4k}^{\#2e}$ in the $c^{sh}$ extension is just the number of odd indexed vertices being incident to both green and orange edges in $C_{4k}$. From the previous lemma, we know that this number has the same parity as the total number of green edges in the cycle $C_{4k}$. As $C_{4k}^{\#2e}$ is a cycle on $2k$ vertices, its total number of edges is an even number, so the total number of edges which cross $I$ should also have the same parity, and so does the difference of the number of edges that cross $I$ in the clockwise direction and anticlockwise direction. This completes the proof.
	\end{proof}
	
	We use $M_{2k}$ to denote the M\"obuis ladder with $2k$ steps. As a graph that is isomorphic to the graph build on $C_{4k}$ by adding an edge between each pair of vertices at a distance $2k$. In the next lemma we show that the  M\"obuis ladder $M_{2k}$ can replace the role of the odd cycle in Lemma~\ref{coro:far-polarC_2k+1^2}. It can then be used similarly to build families of graphs with circular chromatic number at least 4.  
	
	\begin{lemma}\label{lem: mobius ladder}
		For any far-polar mapping $c$ of $M_{2k}$ to $O_r$, the winding number $w(C_{4k}^{\#2e}, c^{sh})$ (or similarly $w(C_{4k}^{\#2o}, c^{sh})$) is odd. 
	\end{lemma}
	
	\begin{proof}
		By Lemma~\ref{lem: 4k zigzag greens and winding number}, it is enough to prove that the number of green edges of $C_{4k}$ in the extension $c^D$ of $c$ is odd. 
		We use the notation $C_{1,2,3,\cdots t}$ for oriented cycle with vertices $v_1,v_2, \cdots v_t$ and directed  edges $\overrightarrow{v_iv_{i+1}}$ for $i \pmod{t}$. We will view $M_{2k}$ as union of $2k$ 4-cycles, see Figure \ref{fig:mobius_4k} for reference.
		
		Consider all oriented $4$-cycles formed by two consecutive steps of ladder, $C_1 : C_{1,2,(2k+2),(2k+1)}$, $C_2 : C_{2,3,(2k+3),(2k+2)}$, $\cdots C_{2k} : C_{2k,(2k+1),1,4k}$. By Lemma~\ref{lem: C_4} we know that each $C_i$'s has two green edges in $c^D$ extension. Therefore, in total, the sum of the number of their green edges is an even number as well.
		To prove our claim, we present a different counting of this number. Consider the oriented $4k$-cycle, $C_{1,2,3,\cdots 4k}$, half of its edges (from $\overrightarrow{v_1v_2}$ to $\overrightarrow{v_{2k}v_{2k+1}}$) agree in orientation with the one in the corresponding $C_i$, but the other half is oriented the opposite direction. So if we want to get back the same orientation as in the $C_i$'s, we should switch $2k$ edges. As changing the orientation of a green edge makes it orange and vice versa, we switch the parity of the number of green edges an even number of times. Now we have to consider the steps of the ladder as well. Except for the edge between $v_1$ and $v_{2k+1}$, every other step $v_iv_{2k+i}$ is oriented as $\overrightarrow{v_{2k+i}v_i}$ in $C_i$ and as $\overrightarrow{v_iv_{2k+i}}$ in $C_{i-1}$ (for $1<i \leq 2k$). So they contribute exactly one green edge (in one of their orientations) to the total sum. The edge between $v_1$ and $v_{2k+1}$ is oriented as $\overrightarrow{v_{2k+1}v_1}$ in both $C_1$ and $C_{2k}$, contributing 0 or 2 to the total sum. Therefore the contribution of the steps is odd in total.
		So, in summary, starting with the oriented $C_{4k}$, changing the orientation of an even number of its edges, and then adding the steps of the ladder, we should get back the same number of green edges as we had in total in the $C_i$'s. Since that is an even number, the oriented $C_{4k}$ must have an odd number of green edges. 
	\end{proof}
	
	\begin{figure}[ht!]
		\centering     
		\begin{tikzpicture}
			\begin{scope}[very thick] 
				
				\draw[line width=1mm, red] (0, 0) -- (2, 0);
				\draw[line width=1mm, red] (0, 2) -- (2, 2);
				\draw[line width=.35mm, red] (0, 0) -- (0, 2);
				\draw[line width=.35mm, red] (2, 0) -- (2, 2);
				
				\draw[line width=1mm, red] (2, 0) -- (4, 0);
				\draw[line width=1mm, red] (2, 2) -- (4, 2);
				\draw[line width=.35mm, red] (4, 0) -- (4, 2);
				
				\draw[line width=1mm, red] (4, 0) -- (6, 0);
				\draw[line width=1mm, red] (4, 2) -- (6, 2);
				\draw[line width=.35mm, red] (6, 0) -- (6, 2);
				
				\draw[line width=1mm, red] (6, 0) -- (8, 0);
				\draw[line width=1mm, red] (6, 2) -- (8, 2);
				\draw[line width=.35mm, red] (8, 0) -- (8, 2);
				
				\draw[line width=1mm, red] (12, 0) -- (14, 0);
				\draw[line width=1mm, red] (12, 2) -- (14, 2);
				\draw[line width=.35mm, red] (14, 0) -- (14, 2);
				\draw[line width=.35mm, red] (12, 0) -- (12, 2);
				
				\draw[very thick, red,dotted] (8,0) -- (12,0);
				\draw[very thick,red,dotted] (8,2) -- (12,2);
				
				\draw[line width=1mm, red] (0, 0) -- (14, 2);
				\draw[line width=1mm, red] (0, 2) -- (14, 0);


				\node at (1,1) {\small $\mathbf{C_1}$};
				\node at (3,1) {\small $\mathbf{C_2}$};
				\node at (5,1) {\small $\mathbf{C_3}$};
				\node at (13,1) {\small $\mathbf{C_{2k}}$};
				\node at (0,2.3) {$v_1$};
				\node at (2,2.3) {$v_2$};
				\node at (4,2.3) {$v_3$};
				\node at (6,2.3) {$v_4$};
				\node at (12,2.3) {$v_{2k-1}$};
				\node at (14,2.3) {$v_{2k}$};
				\node at (0,-0.4) {$v_{2k+1}$};
				\node at (2,-0.4) {$v_{2k+2}$};
				\node at (4,-0.4) {$v_{2k+3}$};
				\node at (6,-0.4) {$v_{2k+4}$};
				\node at (12,-0.4) {$v_{4k-1}$};
				\node at (14,-0.4) {$v_{4k}$};

				\draw [-latex, thick]
				(.6,.6) 
				++(0:4mm)  
				arc (270:00:4mm);
				
				\draw [-latex, thick]
				(2.6,.6) 
				++(0:4mm)  
				arc (270:00:4mm);
				
				\draw [-latex, thick]
				(4.6,.6) 
				++(0:4mm)  
				arc (270:00:4mm);
				
				
				\draw [-latex, thick]
				(12.6,.6) 
				++(0:4mm)  
				arc (270:00:4mm);
				
			\end{scope}
		\end{tikzpicture}
		\caption{M\"obius ladder $M_{2k}$.}
		\label{fig:mobius_4k}
	\end{figure}
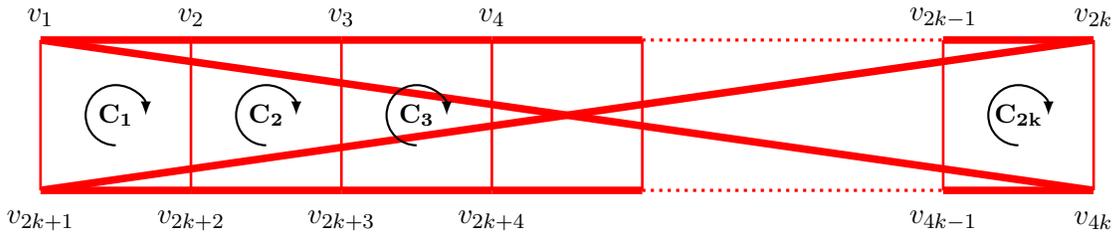
	
	We can now state our theorem for $\widehat{BQ}({\ell},2k)$.
	\begin{theorem}\label{thm:WidningC4}
		For any positive integers $\ell$ and $k$, we have $\chi_c\widehat{BQ}({\ell},2k)=4$.
	\end{theorem}
	
	
	\begin{proof}
		As in the previous cases, we can consider $\widehat{BQ}({\ell},2k)$ as a graph obtained from the $l \times 2k$ cylindrical grid by adding a universal vertex on the first layer and completing the last two layers into a M\"obuis ladder. For any $r < 4$, 
  $r$-coloring would be a far-polar mapping of this graph, which, by Lemma~\ref{lem: mobius ladder} would imply an odd winding number for each of the last layers in the $C^{sh}$ extension, but by  Observation~\ref{obs:AddingUniversalVertex} the first layer has the winding number $0$, but by Lemma~\ref{lem:WC_2k^2} all layers have the same parity of winding number.
	\end{proof}

	Finally, we use this to prove that $\widehat{BM}(\ell, 2k)$ also has the same circular chromatic number.
	
	\begin{theorem}\label{thm:ChiBQlC2k}
		For any positive integers $\ell$ and $k$, we have $\chi_c(\widehat{BM}({\ell},2k))=4$.
	\end{theorem}
	
	\begin{proof}
		As $BM(\ell, 2k)$ is a signed bipartite graph, $4$ is an upper bound on its circular chromatic number.
		To prove that it is also a lower bound, we consider $\widehat{BQ}({\ell}+2,2k)$ and switch at first $2k$ vertices of the M\"obuis ladder built on the first two layers. These are vertices labelled $v_{1,1}, v_{2,1}, v_{1,2}, v_{2,2}, \dots, v_{1,k}, v_{2,k}$. At the end all diagonal edges of the M\"obuis ladder are positive. We consider a homomorphic image of this signed graph by identifying two ends of each diagonal edge of the M\"obuis ladder. That is $v_{1,1}$ is identified with $v_{1,k+1}$, $v_{2,1}$ is identified with $v_{2,k+1}$ and so on. Then, we can identify each $v_{1,i}$ with $v_{3,i+k}$ as well ($i \in \{1,2, \dots k\}$). It can then be verified that the image is the signed graph obtained from $BM(\ell, 2k)$ by adding a positive loop to each of the $k$ vertices $u_j$ and to the next layer. 
		However, a positive loop does not change the circular chromatic number of a signed graph. Thus we have  $\chi_c(BM({\ell},2k))\geq \chi_c(\widehat{BQ}({\ell}+2,2k))=4$.    	
	\end{proof}
	
	We note that by identifying the two ends of each positive edge in $\widehat{BQ}(\ell, 2k+1)$ we get a copy of $M_{\ell}(C_{2k+1})$ together with some positive loops. Thus, in a similar fashion, one can view Theorem~\ref{thm:ChiMlC2k+1} as a corollary of Theorem~\ref{thm:ChiBMl2k+1}. 
	
	\section{Concluding remarks}\label{sec:Conclusions}

	The special subclass of $M_{k}(C_{2k+1})$, on $2k^2+k+1$ vertices, is conjectured in \cite{VT95} to have the smallest number of vertices among 4-chromatic graphs of odd-girth $2k+1$. In \cite{ES18}, this is verified to be the case with an added assumption that every pair of odd cycles share a vertex. For the general case, a lower bound of $(k-1)^2$ for the number of vertices of a 4-critical graph of odd girth $2k+1$ is given in \cite{J01} modifying the method of \cite{N99}. A natural bipartite analogue of this question is to find the smallest number of vertices of a signed bipartite graph of negative girth $2k$ whose circular chromatic number is $4$. Here we gave two families of such graphs, where the graphs of negative girth $2k$ have $2k^2-k+1$ vertices, namely $\widehat{BQ}(k, 2k-1)$ and $\widehat{BM}(k-1, 2k)$. A lower bound of $k^2$, for the order of such signed bipartite graphs is given in an unpublished joint work of the second author with L.A. Pham and Z. Wang. They also improve the lower bound of $(k-1)^2$ to $k^2$ for the case of graphs.
	
	The notion of winding number, or rather parity of it, as a simpler case of the use of topological method in the study of graph coloring, is used in a number of papers. For example it is the main tool in proving that $K_3$ is multiplicative \cite{ES85}. Payan's proof in \cite{P92} of the fact that generalized Mycielski on odd cycles has chromatic number 4 also relies on this technique. However, in most application of the winding number the geometric notion is hidden. One of the advantages of this work then is, by using the notion of circular coloring rather than proper coloring, we formally present the use of winding number for proving coloring properties of graphs. Here we recall how the winding number is used in the literature and show its equivalence to our approach.
	
	Let $C$ be cycle together with a 3-coloring $\psi$. We view this 3-coloring as a circular 3-coloring, thus the 3 colors are three points, say $a$, $b$, and $c$ of a circle mutually at the same distance. We then considers $C^{\#2}$ and its extension by the shortest rule. Thus if the two neighbors $x_{i-1}$ and $x_{i+1}$ of a vertex $x_{i}$ are given the same color by $\psi$, then the whole edge $x_{i-1}x_{i+1}$ of $C^{\#2}$ is mapped to a single point. Otherwise, the three vertices $x_{i-1}$,  $x_{i}$, and $x_{i+1}$ are given three different coloring, (say $a$, $b$, and $c$, respectively), and thus in the shortest extension the edge $x_{i-1}x_{i+1}$ of $C^{\#2}$ is mapped to the shorter arc with ends $a$ and $c$, that is the one which does not include the point $b$. Thus if we were to compute the parity of  the winding number $w(C^{\#2}, \psi)$ extended with the shortest rule, with our green interval being on the shorter side of the arc $ac$, all we need is to count the number of vertices colored $b$ whose two neighbors are colored differently. This number must then be the same as vertices colored $a$ (or $c$) whose neighbors are of different colors. If the length of the cycle itself is odd, then parity of the total number of vertices whose neighbors are colored differently is the same as the parity of the number of vertices colored $a$ whose neighbors are colored differently. This total number is then the tool commonly used in the literature.


	{\bf Acknowledgment.} This work is supported by the following grants and projects: 1. ANR-France project HOSIGRA (ANR-17-CE40-0022). 2. Indo-French Center of Applied Mathematics, project AGRAHO “Applications of graph homomorphisms”(MA/IFCAM/18/39). 3. Math-AmSud project PLANNING. 4. National Research, Development and Innovation Office (NKFIH) grant K--120706 of NKFIH Hungary. 5. WLI grant(SB22231494MAIITM008570) of IIT Madras, India. The second author would also like to thank Lan Ann Pham and Zhouningxin Wang for earlier discussions on this subject.

	\bibliographystyle{plain}
	\bibliography{Mycielski}

\end{document}